\documentclass{article}
\usepackage[utf8]{inputenc}
\usepackage{mathtools, nccmath}
\usepackage{geometry}
\usepackage{xcolor}
\usepackage{amsmath}
\usepackage{multirow}
\usepackage{amsthm}
\usepackage{changepage}
\usepackage{amssymb}
\usepackage{bbm}
\usepackage{xcolor}
\usepackage{graphicx}
\usepackage{algorithm,algorithmic}
\usepackage{cases}
\usepackage{enumitem}
\usepackage{subfig}
\usepackage[export]{adjustbox}
\usepackage[english]{babel}
\DeclareMathOperator*{\softmin}{soft-min}

\newtheorem{theorem}{Theorem}[section]

\newtheorem{lemma}[theorem]{Lemma}
\newtheorem{proposition}[theorem]{Proposition}
\newtheorem{assumption}[theorem]{Assumption}
\newtheorem{remark}[theorem]{Remark}
\newtheorem{definition}[theorem]{Definition}

\DeclareMathOperator*{\argmax}{arg\,max}
\DeclareMathOperator*{\argmin}{arg\,min}

\DeclareFontFamily{U}{mathx}{\hyphenchar\font45}
\DeclareFontShape{U}{mathx}{m}{n}{
      <5> <6> <7> <8> <9> <10>
      <10.95> <12> <14.4> <17.28> <20.74> <24.88>
      mathx10
      }{}
\DeclareSymbolFont{mathx}{U}{mathx}{m}{n}
\DeclareFontSubstitution{U}{mathx}{m}{n}
\DeclareMathAccent{\widecheck}{0}{mathx}{"71}
\DeclareMathAccent{\wideparen}{0}{mathx}{"75}

\usepackage[colorlinks,citecolor=blue]{hyperref}
\title{Analysis of Multiscale Reinforcement Q-Learning Algorithms for Mean Field Control Games
}
\author{
Andrea Angiuli\thanks{Prime Machine Learning Team, Amazon. 320 Westlake Ave N, SEA83, Seattle, WA, 98109 (E-mail: \href{mailto:aangiuli@amazon.com}{aangiuli@amazon.com}). The work presented here does not relate to this author's position at Amazon. }
  \and Jean-Pierre Fouque\thanks{Department of Statistics and Applied Probability, South Hall, University of California, Santa Barbara, CA 93106, USA (E-mail: \href{mailto:fouque@pstat.ucsb.edu}{fouque@pstat.ucsb.edu}). Work supported by NSF grant DMS-1953035.} 
  \and Mathieu Lauri\`ere\thanks{Shanghai Frontiers Science Center of Artificial Intelligence and Deep Learning; NYU-ECNU Institute of Mathematical Sciences at NYU Shanghai; NYU Shanghai, 567 West Yangsi Road, Shanghai, 200126, People’s Republic of China (E-mail: 
  \href{mailto:mathieu.lauriere@nyu.edu}{mathieu.lauriere@nyu.edu}).}
  \and  Mengrui Zhang\thanks{Department of Statistics and Applied Probability, South Hall, University of California, Santa Barbara, CA 93106, USA (E-mail: \href{mailto:mengrui@umail.ucsb.edu}{mengrui@umail.ucsb.edu}).
  }}
\begin{document}

\maketitle
\centerline{\textbf{Abstract}}
\begin{adjustwidth}{50pt}{50pt} 
Mean Field Control Games (MFCG), introduced in \cite{andrea22}, represent competitive games between a large number of large collaborative groups of agents in the infinite limit of number and size of groups. In this paper, we prove the convergence of a three-timescale Reinforcement Q-Learning (RL) algorithm to solve MFCG in a model-free approach from the point of view of representative agents. Our analysis uses a Q-table for finite state and action spaces updated at each discrete time-step over an infinite horizon. In \cite{andrea23}, we proved convergence of two-timescale algorithms for MFG and MFC separately highlighting the need to follow multiple population distributions in the MFC case. Here, we integrate this feature for MFCG as well as three rates of update decreasing to zero in the proper ratios. Our technique of proof uses a generalization to three timescales of the two-timescale analysis in \cite{Borkar97}. We give a simple example satisfying the various hypothesis made in the proof of convergence and  illustrating the performance of the algorithm. 
\end{adjustwidth}

\section{Introduction}

\subsection{Background}

Reinforcement learning (RL) is a set of methods to solve complex problems without knowing the full knowledge of the model. We refer e.g. to~\cite{sutton2018reinforcement}
 for more background on this topic. While RL is traditionally focused on solving discrete time optimal control problems, also called Markov decision processes (MDP), the topic has recently attracted the interest of the stochastic optimal control community and several connections with continuous time optimal control methods have been made, see e.g.~\cite{wang2020reinforcement,wang2020continuous}. Besides optimal control problem, RL methods are also used to solve games. However, games with a large number of players are challenging for traditional methods, including multi-agent reinforcement learning methods. In order to scale up in terms of number of agents, mean field games (MFGs) have been introduced by~\cite{MR2295621} and~\cite{MR2346927}. Such games provide approximate Nash equilibria for finite-player games by passing to the limit and studying the interactions between a representative player and a mean field representing the rest of the population. RL methods for MFGs have been developed, see e.g.~\cite{SubramanianMahajan-2018-RLstatioMFG,guo2019learning,elie2020convergence,cui2021approximately} among many others. While MFGs typically use the notion of Nash equilibria, which corresponds to non-cooperative players, one can also consider an infinite population of cooperative players, which leads to the notion of solution of social optimum or mean field control (MFC). MFC problems have also been studied from the RL viewpoint, see e.g.~\cite{CarmonaLauriereTan-2019-LQMFRL,gu2021meanQ,motte2019mean,carmona2023model,frikha2023actor}. We refer the interested reader to~\cite{MR3134900} and~\cite{carmona2018probabilisticI-II} for more details on the theoretical background on MFGs and MFC problems, and to~\cite{lauriere2022learning} for a survey on learning in MFGs. 

Although the classical setting of MFGs and MFC problems efficiently addresses the question of scalability in terms of population size, these two classes of problems focus on extreme situations where all the agents are either non-cooperative or cooperative. However, many situations involve a combination of cooperation and competition. For example, in some situations, there is a competition between several large coalitions. Such problems have been studied in the framework of mean-field type games~\cite{tembine2017mean}. When the number of coalitions is also large, the problem can be approximated using the framework of mean field control games (MFCG) introduced in~\cite{andrea22}. The main goal of the present paper is to provide a proof of convergence for the RL algorithm proposed in~\cite{andrea22} to solve MFCGs.

This algorithm builds on several previous works. In \cite{Andrea20}, the authors proposed a unified RL algorithm to solve MFG and MFC problems in the context of finite state and action spaces, and in discrete infinite time horizon. The algorithm depends on the ratio of the two learning rates, one for the Q-function and the other for the distribution of the population, leading to a two timescale algorithm. In \cite{andrea23}, we proved the convergence of this algorithm using a generalization of the  two-timescale approach of \cite{Borkar97}. This careful mathematical analysis revealed that the MFC case requires the updating of one population distribution for each state-action point which led us to consider two different algorithm, one for MFG and one for MFC. On the other hand,  \cite{andrea22} revealed that a single algorithm with three timescales could handle MFCGs, which are a mix of MFG and MFC. They naturally model competitive games between a large number of large collaborative groups of agents. In the limit of the number of groups and the sizes of groups going to infinity, we showed in  \cite{andrea22}  that MFCG provides an approximation of the corresponding finite player game. The unified three-timescale RL algorithm that we analyse in this paper updates the Q-function of a representative agent learning the optimal policy at equilibrium, the global population distribution, and one local population distribution for each state-action point. The local populations are updated faster the Q-function which in turns is updated faster than the global population.

\subsection{Structure of the Paper}

In Section \ref{sec:notations}, we recall the classical Q-learning setup and we introduce the notations needed for the Mean Field Control Game (MFCG) problem presented in Section \ref{def:MFCG}. In Section \ref{sec:algos}, we present the algorithms studied in the paper. We start in Section \ref{sec:fullalgo} with the full asynchronous algorithm with stochastic approximation and explain that it is {\it model free} from the point of view of representative agents, one for each population involved in the system. We also highlight the three timescales involved, one fast for the local population distributions (MFC aspect), one slow for the global population (MFG aspect), and one in between for the update of the Q-table used to choose the policy.  In Section \ref{sec:syncalgo}, we simplify the algorithm by allowing to update synchronously all the entries of the Q-table at each time step. We further simplify it in Section \ref{sec:idealalgo} by assuming that expectations are known without stochastic approximation leading to a deterministic system of iterations. In Section \ref{sec:convergence}, we derive the convergence of our algorithms, starting with ideal deterministic iteration scheme, then generalized to include successively  stochastic approximation and asynchronous features. The proofs are based on generalizations to three timescales the results of \cite{Borkar97}.  We end in Section \ref{sec:example} with numerical illustrations on a very simple example with two states and two actions.

\section{Model and MFCG Formulation}\label{sec:notations}

\subsection{Model and Notations}

First, we present classical Q-learning in a nutshell.
\paragraph{Classical Q-learning. }  Classical RL aims at solving a Markov Decision Process (MDP) using the following setting. At each discrete time $n$, the agent observes her state  $X_n$ in a finite state space $\mathcal{{X}} =\{ x_0, \dots, x_{|\mathcal{{X}}|-1}\}$, and, based on it, chooses an action $A_n$ in a finite action space $\mathcal{{A}} =\{ a_0, \dots, a_{|\mathcal{{A}}|-1}\}$. Then, the environment evolves and provides the agent with a new state $X_{n+1}$ and reports a reward  $r_{n+1}$. The goal of the agent is to find the optimal policy $\pi$ such that it assigns to each state the optimal probability of actions in order to maximize the expected cumulative reward. The problem can be recast as the problem of learning the optimal state-action value function, also called Q-function: $Q^\pi(x,a)$ represents the expected cumulative discount rewards when starting at state $x$, using an action $a$, and then following policy $\pi$. Mathematically, 
$$
    Q^\pi(x,a) = \mathbbm{E}\left[\sum_{n=0}^{\infty}\gamma^n r_{n+1}|X_0 = x, A_0=a\right],
$$
where $r_{n+1} = r(X_n,A_n)$ is the instantaneous reward, $\gamma\in(0,1)$ is a discounting factor, $X_{n+1}$ is distributed according to a transition probability which depends on $X_n$ and the action drawn from $\pi(X_n)$. The goal is to compute the optimal Q-function defined as:
\[
    Q^*(x,a) = \max_{\pi} Q^\pi(x,a),
\]
which in turns gives the optimal control $\alpha^*(x):=\argmax_aQ^*(x,a)$ if unique. 

To this end, the Q-learning method was introduced by \cite{watkins1989learning}. The basic idea is to iteratively sample an action $A_n \sim \pi(X_n)$ according to a behavior policy $\pi$, observe the induced state $X_{n+1}$ and reward $r_{n+1}$, and then update the Q-table according to the formula: 
\[
    Q(X_n,A_n)
        \gets Q(X_n,A_n) +  \rho\left[r_{n+1} + \gamma\max_{a'\in\mathcal{A}}Q(X_{n+1},a')-Q(X_n,A_n)\right]
\]
where $\rho \in (0,1)$ is a learning rate. 

Note that in our case, to be consistent with most of the MFG and MFC literature, we will minimize costs instead of maximizing rewards.

\paragraph{Notations.}
The following notations will be used  in the context of MFCG studied in this paper.

As in the case of classical RL we denote by  $\mathcal{{X}} =\{ x_0, \dots, x_{|\mathcal{{X}}|-1}\}$ and $\mathcal{{A}} =\{ a_0, \dots, a_{|\mathcal{{A}}|-1}\}$ the  finite state and action spaces.  We denote by $\boldsymbol{\delta}$ the indicator function, i.e., for any $x \in \mathcal{X}$, $\boldsymbol{\delta}(x) = [\mathbbm{1}_{x_0}(x),\dots,\mathbbm{1}_{x_{|\mathcal{{X}}|-1}}(x)]$, which is a vector full of $0$ except at the coordinate corresponding to $x$. 
Let $\Delta^{|\mathcal{X}|}$ be  the simplex of probability measures on $\mathcal{X}$. Let $p$ : $\mathcal{X} \times \mathcal{X} \times \mathcal{A} \times \Delta^{|\mathcal{X}|} \times \Delta^{|\mathcal{X}|}\to \Delta^{|\mathcal{X}|}$ be a transition kernel which depends on both a global distribution $\mu$ and  local distribution $\mu'$. We also interpret the kernel $p$ as a function
\begin{align*}
    p : \mathcal{X} \times \mathcal{X} \times \mathcal{A} \times \Delta^{|\mathcal{X}|} \times \Delta^{|\mathcal{X}|}\to [0,1], \hspace{20pt}(x,x',a,\mu,\mu')\mapsto p(x'|x,a,\mu,\mu'),
\end{align*}
which provides the probability that the process jumps to the states $x'$ given the current state is $x$, the action $a$ is taken, the global distribution is $\mu$ and the local distribution is $\mu'$.
A policy $\pi\in \Pi$ is a collection of probability distributions $\{\pi(x)=\pi(\cdot|x), x\in \mathcal{X}\}$ on the set of actions $\mathcal{A}$, that is $\pi: \mathcal{X}\to\Delta^{|\mathcal{A}|}$.
We  denote by $\mathrm{P}^{{\pi},{\mu},\mu'}$ the transition kernel according to the global distribution ${\mu}$ and local distribution $\mu'$ on $\mathcal{X}$, and the policy ${{\pi}}\in \Pi$,  defined for any  distribution $\tilde\mu \in \Delta^{|\mathcal{X}|}$ by:
\begin{equation}
\label{eq:def-Ptransitions}
\tilde\mu\mathrm{P}^{\pi,\mu,\mu'} (x) = \sum_{x'\in \mathcal{X}} \tilde\mu(x') \sum_{a\in \mathcal{A}} \pi(a|x') p(x|x',a,\mu,\mu'), \qquad x \in \mathcal{X}.
\end{equation}
We denote by $\mu^{\pi, \mu}$ the asymptotic local distribution of the controlled process following the strategy $\pi$ when the global distribution is $\mu$. We assume  ergodicity of such processes so that $\mu^{\pi, \mu}$ exists and is unique. 

Let $f$: $\mathcal{X} \times \mathcal{A} \times \Delta^{|\mathcal{X}|}\times \Delta^{|\mathcal{X}|} \to \mathbbm{R}$ be a running cost function. We interpret $f(x,a,\mu,\mu')$ as the one-step cost, at any given time step, incurred to a representative agent who is at state $x$ and uses action $a$ while the population distributions are $\mu$ and $\mu'$.
For a policy $\pi$ we denote
\[
f(x,\pi,\mu,\mu')= \sum_{a\in \mathcal{A}} \pi(a|x) f(x,a,\mu,\mu')
\]
Based on these notations, we present the notion of Mean Field Control Game (MFCG) introduced in \cite{andrea22}.

\subsection{MFCG Definition} \label{def:MFCG}

Using the notations introduced in the previous section,
given a running cost function $f$, a discount rate $\gamma<1$, and an initial distribution $\mu_0\in \Delta^{|\mathcal{X}|}$, we consider the following {\bf infinite horizon asymptotic formulation of a mean field control game problem}:

\begin{definition}\label{mfcg_definition}
An MFCG problem consists in finding a policy $\hat{\pi}\in \Pi$ and a distribution $\hat{\mu}\in \Delta^{|\mathcal{X}|}$ such that:

    1. (best response) $\hat{\pi}$ is the minimizer of the value function
\begin{align*}
    \pi \mapsto V^{\pi}_{\hat{\mu}} := E\left[\sum_{n=0}^{\infty}\gamma^n f(X^{\pi,\hat{\mu}}_n,\pi(X^{\pi,\hat{\mu}}_n),\hat{\mu},\mu^{\pi,\hat\mu})\right],
\end{align*}
with $\mu^{\pi,\hat{\mu}} = \lim_{n\to\infty}\mathcal{L}(X^{\pi,\hat{\mu}}_n)$ where the process $(X^{\pi,\hat{\mu}}_n)_{n\geq 0}$ follows the dynamics
\begin{align*}
    \begin{cases}
        X^{\pi,\hat{\mu}}_0 \sim \mu_0
        \\
        P(X^{\pi,\hat{\mu}}_{n+1}=x'|X^{\pi,\hat{\mu}}_n=x,A_n=a,\mu=\hat{\mu}, \mu' =\mu^{\pi,\hat{\mu}} ) 
    = p(x'|x,a,\hat{\mu},\mu^{\pi,\hat{\mu}})\\
     A_n \sim \pi(\cdot|X_n^{\pi,\mu})\quad \mbox{independently at each time}\,\, n\geq 0.
    \end{cases}
\end{align*}

2. (consistency) $\hat{\mu}= \lim_{n\to\infty}\mathcal{L}(X^{\hat{\pi},\hat{\mu}}_n)=\mu^{\hat{\pi},\hat{\mu}}$.
\end{definition}

In order to make sense of the above problem statement we have to restrict to policies $\pi: \mathcal{X}\to\Delta^{|\mathcal{A}|}$ which are such that for any $\mu$ the controlled process $X^{\pi,\mu}_n$ has a limiting distribution, i.e. $\lim_{n\to\infty} \mathcal{L}(X^{\pi,\mu}_n)$ exists. For a finite state Markov chain this is the case if $(X^{\pi,\mu}_n)_n$ is irreducible and aperiodic. We therefore assume that the strategy $\hat{\pi}$ is the minimizer over all strategies such that $(X^{\pi,{\mu}}_n)_{n}$ is irreducible and aperiodic for all $\mu$. 

\subsubsection{Q-learning for MFCG}\label{sec:Q-MFCG}

 We  adapt the $Q$-learning concepts to the MFCG problem at hand. Focusing on the problem faced by a representative agent, we note that the local distribution is not fixed and depends on the policy itself. Thus, we have to adapt the classical $Q$-learning in the spirit of \cite{Andrea20}. For an admissible policy $\pi \in \Pi$ and a pair $(x,a) \in \mathcal{X}\times\mathcal{A}$, we define the new control $\tilde\pi^{(x,a)}$ by 
\begin{align}
    \tilde{\pi}^{(x,a)}(x') := \begin{cases}
        \delta_a\quad &\text{if}\hspace{5mm} x'=x, \\
        \pi(x')\quad &\text{for}\hspace{3mm} x'\neq x.
    \end{cases} 
\end{align}
Given a global distribution $\mu$ and a policy $\pi$, the $Q$-function for our problem is given by:
\[
    Q^{\pi}_{\mu} (x,a)=f(x,a,\mu,\mu^{\tilde\pi^{(x,a)},\mu}) + \mathbb{E}\left[\sum_{n=1}^{\infty}\gamma^{n}f\left(X^{\pi,\mu}_{n},\pi(X^{\pi,\mu}_{n}),\mu,\mu^{\tilde\pi^{(x,a)},\mu}\right)\Big\lvert X^{\pi,\mu}_{0}=x,A_{0}=a\right],
\]
where $\mu^{\tilde\pi^{(x,a)},\mu}$ is the limiting local distribution relative to the policy $\tilde\pi^{(x,a)}$ and when the global distribution is $\mu$.

The optimal function in the sense of minimizing cost is given by
$$
Q_{\mu}^{*} (x,a) := \min_{\pi} Q_{\mu}^{\pi}(x,a).
$$ 
Note that the minimizing strategy $\pi^*$ may depend on the global measure $\mu$.

A simple generalization of Theorem 2 in Appendix C in \cite{Andrea20}, gives the following Bellman equation satisfied by $Q_{\mu}^{*}$ for a fixed $\mu$:
\begin{equation}\label{MKV-Bellman}
Q_{\mu}^{*}(x,a)=f(x,a,\mu,\mu^{\widetilde{\pi^*}^{(x,a)},\mu})+\gamma\sum_{x'}p\left(x'\lvert x,a,\mu,\mu^{\widetilde{\pi^*}^{(x,a)},\mu}\right)\min_{a'}Q_{\mu}^{*}(x',a'),
\end{equation}
where $\mu^{\widetilde{\pi^*}^{(x,a)},\mu}$ is the limiting local distribution relative to the policy $\widetilde{\pi^*}^{(x,a)}$ and when the global distribution is $\mu$. Note that the advantage of introducing one local distribution for each state-action point $(x,a)$ is that the Q-table remains a function of state and action.

Finally, we introduce the specific Bellman operator $ \mathcal{B}$ given by
\begin{equation}\label{eq:BQ}
    (\mathcal{B}_{\mu,\mu'} Q)(x,a) := f(x,a,\mu,\mu')+\gamma \sum_{x'\in\mathcal{X}}p(x'|x,a,\mu,\mu')\min_{a'}Q(x',a'),
\end{equation}
for two distributions $\mu$ and $\mu'$, so that \eqref{MKV-Bellman} reads $Q^*_\mu(x,a)=\mathcal{B}_{\mu,\mu^{\widetilde{\pi^*}^{(x,a)},\mu}}Q^*_\mu(x,a)$, or, in a lighter notation, $Q^*_\mu=\mathcal{B}_{\mu,\mu^*}Q^*_\mu$ with $\mu^*=\mu^{\widetilde{\pi^*}^{(x,a)},\mu}$ where we have dropped the $(x,a)$-dependence.

\section{Algorithms and Multiscale Learning Rates}\label{sec:algos}

First, in Section \ref{sec:fullalgo}, we present our full three-timescale algorithm which updates the Q-table and the population distributions in an asynchronous way (one $(x,a)$ at a time along one path of a multidimensional controlled process) and uses a stochastic approximation. Then, in Section \ref{sec:syncalgo}, we consider a simplified version of the full algorithm which updates the Q-table and the population distributions synchronously (all $(x,a)$ at a time using sampling from the law of the controlled process) and still using stochastic approximation. Finally, in Section \ref{sec:idealalgo}, we replace the stochastic approximation by a perfect knowledge of the expectations involved in the time iterations, leading to a three-timescale ideal deterministic algorithm.

Our proof of convergence in Section \ref{sec:convergence}, will proceed the other way around by first proving convergence of the ideal deterministic algorithm by using a three-timescale generalization of two-timescale results in \cite{Borkar97}, and then successively restore the stochastic approximation  and the asynchronous features by controlling the corresponding errors.

\subsection{Full Algorithm}\label{sec:fullalgo}

We present an improved version of the  three-timescale Algorithm 1 introduced in \cite{andrea22}. The difference  is that it considers several population distributions and paths of multidimensional control processes which are essential features in the derivation of convergence results, the focus of this paper.

Algorithm~\ref{algo:U3MFQL} below shows the pseudo-code for our three-timescale MFCG Q-learning algorithm. It involves 
three learning rates that we choose of the following form
\begin{align*} %
    \rho^{\mu}_{n}=\frac{1}{(1+n)^{\omega^{\mu}}}\hspace{20pt}
    \rho^Q_{n,x,a}=\frac{1}{(1+\nu(x,a,n))^{\omega^Q}}\hspace{20pt} \rho^{\tilde{\mu}}_{n,x,a}=\frac{1}{(1+\nu(x,a,n))^{\omega^{\tilde{\mu}}}}
\end{align*}
where $\omega^\mu$, $\omega^Q$ and $\omega^{\tilde\mu}$ are parameters
such that $\frac{1}{2}<\omega^{\tilde{\mu}}<\omega^Q<\omega^{\mu}<1$
, and $\nu(x,a,n)$ is the number of times a controlled process $(X,A)$ visits  $(x,a)$ up to time $n$, that is, $\nu(x,a,n)=\sum_{m=0}^{n}\mathbbm{1}_{\{(X_m,A_m)=(x,a)\}}$.

\begin{algorithm}[H]
\caption{Three-timescale Mean Field Control Game Q-learning Algorithm\label{algo:U3MFQL}} 
\begin{algorithmic}[1]
\REQUIRE 
$N_{steps}$: number of steps; $\phi$: parameter for $\softmin$ policy; $(\rho^Q_{n,x,a})_{n,x,a}$: learning rates for the value function; $(\rho^\mu_{n})_{n}$: learning rate for the global distribution; $(\rho^{\tilde\mu}_{n,x,a})_{n,x,a}$: learning rates for the local distribution.
\STATE \textbf{Initialization}: $Q_{0}(x,a) =0$ for all $(x,a)\in \mathcal{{X}}\times \mathcal{{A}}$, $\mu_{0}=\mu_0^{(x,a)}=[\frac{1}{|\mathcal{{X}}|},\dots,\frac{1}{|\mathcal{{X}}|}] $

\STATE \textbf{Observe} $X_{0}\sim\mu_0$, $X^{(x,a)}_0\sim\mu_0^{(x,a)}$ \\

\STATE
{\textbf{Update distributions: } 
$\mu_{0} = \boldsymbol{\delta}(X_{0})$ , $\mu_0^{(x,a)}=\boldsymbol\delta(X^{(x,a)}_0)$}
\\

\FOR{$n=0,1,2,\dots, N_{steps}-1$}

\STATE \textbf{Choose action } $A_{n} \sim \softmin_\phi Q_n(X_{n},\cdot)$ and observe state $X_{n+1} \sim p(\cdot|X_{n}, A_{n}, \mu_n, \mu^{(X_n,A_n)}_n)$ and cost $f_{n+1}=f(X_{n},A_{n},\mu_{n},\mu^{(X_n,A_n)}_{n})$ provided by the environment %
\\
\STATE \textbf{Choose action } Choose $A^{(x,a)}_n = a$ if $X_n^{(x,a)} =x$, otherwise $A^{(x,a)}_{n} \sim \softmin_\phi Q_n(X^{(x,a)}_{n},\cdot)$ and observe state $X^{(x,a)}_{n+1} \sim p(\cdot|X^{(x,a)}_{n}, A^{(x,a)}_{n},\mu_n,\mu^{(x,a)}_{n})$\\

\STATE\textbf{Update local distributions: } $ \mu^{(x,a)}_{n+1}
        = \mu^{(x,a)}_n +  \rho^{\tilde\mu}_{n,X^{(x,a)}_n,A^{(x,a)}_{n}}(\boldsymbol{\delta}(X^{(x,a)}_{n+1}) - \mu^{(x,a)}_{n})$\\
\STATE\textbf{Update global distribution: } $\mu_{n+1} = \mu_{n} + \rho^{\mu}_{n} (\boldsymbol{\delta}(X_{n+1}) - \mu_{n})$
\\

\IF{$X^{(X_n,A_n)}_{n} = X_n$} %

    \STATE
    \textbf{Update value function: } $Q_{n+1}(x,a) = Q_{n}(x,a)$ for all $(x,a) \neq (X_n,A_n)$, and 
        \[
        Q_{n+1}(X_n,A_n)
            = Q_n(X_n,A_n) +  \rho^Q_{n,X_n,A_n}[f_{n+1} + \gamma\min_{a'\in\mathcal{A}}Q_n(X_{n+1},a')-Q_n(X_n,A_n)]
        \]
\ENDIF

\ENDFOR
\STATE {\textbf{Return}} $(\mu_{N_{steps}},\mu^{(x,a)}_{N_{steps}},Q_{N_{steps}})$
\end{algorithmic}
\end{algorithm}

Our algorithm uses  the transition distribution $p$ and the cost function $f$ (line 5), only through samples and evaluations respectively. In other words, the algorithm is \emph{model-free} from the point of view of the representative agent. Since it only uses samples and not expected values, such updates are sometimes referred to as \emph{sample-based updates}. Furthermore, at time-step $n$ in line 10, the value function $Q_{n+1}$ differs from $Q_{n}$ at only one state-action pair, so the algorithm is \emph{asynchronous}. 

Note  that the algorithm differentiates
the local population distributions $\mu_n^{(x,a)}$  updated on line 7 from the global population distribution  $\mu_n$  updated on line 8, by the choice of the learning rate parameters $\omega^{\tilde{\mu}}$ and $\omega^\mu$ satisfying  $\omega^{\tilde{\mu}}<\omega^Q<\omega^{\mu}$. In other words, the local distribution are updated faster than the Q-table providing the control at each time (MFC aspect), while the global distribution is update slower than the Q-table (in the spirit of MFG).

Observe that the choice of actions on lines 5 and 6 uses a $\softmin_\phi$ regularized version of $\argmin$ needed in the proof of convergence in Section \ref{sec:convergence}. We recall that for a positive parameter $\phi$,
\begin{align*}\label{softmin}
{\softmin}_\phi(z)=\left(\frac{e^{-\phi z_i}}{\sum_{j}e^{-\phi z_j}}\right)_{i=1,2,...,|\mathcal{A}|}
\quad \mbox{for}\quad z\in \mathbbm{R}^{|\mathcal{A}|}.
\end{align*}

\subsection{Synchronous Algorithm with Stochastic Approximation}\label{sec:syncalgo}

In this section we modify Algorithm \ref{algo:U3MFQL} in the a synchronous setting with stochastic approximation. This is an intermediate step between the practical setting (asynchronous and sample-based updates), and the ideal setting (synchronous and expectation-based updates). Let us assume that for any $(x,a,\mu,\tilde{\mu})$, the learner can know the value $f(x,a,\mu,\tilde{\mu})$. Furthermore, the learner can sample realizations of the random variables
\begin{align*}
    X^{'}_{x,\pi,\mu,\tilde{\mu}}\sim p(\cdot|x,\pi,\mu,\tilde{\mu}),
\end{align*}
where the action is sampled from a policy $\pi$.
Typically the policy $\pi$ will be $\softmin_\phi Q_n(X_n)$ or $\tilde{\pi}^{(x,a)}$ defined by
\begin{align}
    \tilde{\pi}^{(x,a)}(x') := \begin{cases}
        \delta_a\hspace{3mm} &\text{if}\hspace{5mm} x'=x, \\
        \softmin_\phi Q(X_n^{(x,a)})\quad &\text{for}\hspace{3mm} x'\neq x.
    \end{cases} 
\end{align}
Then, the learner has access to realizations of the following random variables $\widecheck{\mathcal{T}}_{\mu,Q,\tilde{\mu}}(x,a)$ and $\widecheck{\mathcal{P}}_{x,\pi}(\nu)$ taking values respectively in $\mathbbm{R}$ and $\Delta^{|\mathcal{X}|}$:
\begin{align*}
    \begin{cases}
\widecheck{\mathcal{T}}_{\mu,Q,\tilde\mu}(x,a) = f(x,a,\mu,\tilde\mu)+\gamma\min_{a'}Q(X^{'}_{x,a,\mu,\tilde\mu},a')-Q(x,a)\\
        (\widecheck{\mathcal{P}}_{x,\pi}(\nu)(x'))_{x'\in\mathcal{X}} = \left(\mathbbm{1}_{\{X^{'}_{x,\pi,\mu,\tilde\mu}=x'\}}-\nu(x')\right)_{x'\in\mathcal{X}}.
    \end{cases}
\end{align*}
Observe that 
\begin{align}
        \mathbb{E}[\widecheck{\mathcal{T}}_{\mu,Q,\tilde\mu}(x,a)] = \sum_{x''}p(x''|x,a,\mu,\tilde\mu)\left[f(x,a,\mu,\tilde\mu)+\gamma\min_{a'}Q(x'',a')-Q(x,a)\right] =: \mathcal{T}_3(\mu,Q,\tilde\mu)(x,a),\label{def:T3}
        \end{align}
        which defines the operator $\mathcal{T}_3$.
  Likewise, we have      
        \begin{align*}
            \mathbb{E}[\widecheck{\mathcal{P}}_{x,\pi}(\nu)(x')]=\sum_{x''}p(x''|x,\pi,\mu,\tilde\mu)(\mathbbm{1}_{\{x''=x'\}}-\nu(x')) = p(x'|x,\pi,\mu,\tilde\mu) - \nu(x'), 
\end{align*}
so that
\begin{align}
    \mathbb{E}[\widecheck{\mathcal 
 {P}}_{X,{\softmin}_\phi Q(X)}(\mu)(x')] &=\sum_{x}\mu(x)\sum_{x''}p(x''|x,
{\softmin}_\phi Q(x),\mu,\tilde\mu)(\mathbbm{1}_{\{x''=x'\}}-\mu(x'))\nonumber\\
    &=:\mathcal{P}_3(\mu,Q,\tilde\mu)(x'), \label{eq:exp-Pcheck-P}\\
    \mathbb{E}[\widecheck{\mathcal 
 {P}}_{X,\tilde\pi(X)}(\mu^{(x,a)})(y)] &=\sum_{x'\neq x}\mu^{(x,a)}(x')\sum_{x''}p(x''|x',{\softmin}_\phi Q(x'),\mu,\mu^{(x,a)})(\mathbbm{1}_{\{x''=y\}}-\mu^{(x,a)}(y))\nonumber\\
    &\quad + \mu^{(x,a)}(x)\sum_{x''}p(x''|x,a,\mu,\mu^{(x,a)})(\mathbbm{1}_{\{x''=y\}}-\mu^{(x,a)}(y))\nonumber\\
    &=\sum_{x'\neq x}\mu^{(x,a)}(x')p(y|x',{\softmin}_\phi Q(x'),\mu,\mu^{(x,a)}) +\mu^{(x,a)}(x)p(y|x,a,\mu,\mu^{(x,a)}) - \mu^{(x,a)}(y)\nonumber\\
    &=:\mathcal{\tilde P}^{(x,a)}_3(\mu,Q,\mu^{(x,a)})(y), \label{eq:mfc-Pcheck-P}
\end{align}
which defines $\mathcal{P}_3(\mu,Q,\tilde\mu)$ in \eqref{eq:exp-Pcheck-P} and 
$\mathcal{\tilde P}^{(x,a)}_3(\mu,Q,\mu^{(x,a)})$ in \eqref{eq:mfc-Pcheck-P}.

Our {\bf synchronous algorithm} is obtained from Algorithm  \ref{algo:U3MFQL} by replacing:
\begin{itemize}
    \item\label{item1}
line 7 by 
$\mu^{(x,a)}_{n+1}
        = \mu^{(x,a)}_n +  \rho^{\tilde\mu}_{n}\widecheck{\mathcal 
 {P}}_{X^{(x,a)}_n,\tilde\pi^{(x,a)}(X^{(x,a)}_n)}(\mu^{(x,a)})$,
  \item \label{item2} 
line 8 by
$\mu_{n+1} = \mu_{n} + \rho^{\mu}_{n} \widecheck{\mathcal 
 {P}}_{X_n,{\softmin}_\phi Q_n(X_n)}(\mu)$,
\item\label{item3}
and line 10 by 
 $Q_{n+1}(x,a)
        = Q_n(x,a) +  \rho^Q_{n}\widecheck{\mathcal{T}}_{\mu_n,Q_n,\mu_n^{(x,a)}}(x,a)$
\end{itemize}
using the deterministic updating rates:
\begin{align}\label{deterministicrates} 
    \rho^{\mu}_{n}=\frac{1}{(1+n)^{\omega^{\mu}}}\ll
    \rho^Q_{n}=\frac{1}{(1+n)^{\omega^Q}}\ll \rho^{\tilde{\mu}}_{n}=\frac{1}{(1+n)^{\omega^{\tilde{\mu}}}},
\end{align}
where $\omega^\mu$, $\omega^Q$ and $\omega^{\tilde\mu}$ are the same parameters as before
such that $\frac{1}{2}<\omega^{\tilde{\mu}}<\omega^Q<\omega^{\mu}<1$. Then our algorithm becomes synchronous as the Q-table is updated at all entries $(x,a)$ at each time $n$. There is still a stochastic approximation which will be handled in the proof of convergence in Section \ref{sec:convergence} by the using the martingale difference property of the differences
\begin{align}\label{mdifference}
\begin{cases}
    \mathbf{P}^{(x,a)}_n&:=\widecheck{\mathcal 
 {P}}_{X^{(x,a)}_n,\tilde\pi^{(x,a)}(X^{(x,a)}_n)}(\mu^{(x,a)})
 -\mathcal{\tilde P}^{(x,a)}_3(\mu_n,Q_n,\mu_n^{(x,a)}), 
\\ 
\mathbf{P}_n&:= \widecheck{\mathcal 
 {P}}_{X_n,{\softmin}_\phi Q_n(X_n)}(\mu)
 -\mathcal{P}_3(\mu_n,Q_n,\mu^{(x,a)}_n), 
\\ 
\mathbf{T}_n(x,a)&:=\widecheck{\mathcal{T}}_{\mu_n,Q_n,\mu^{
(x,a)}_n}(x,a)
-\mathcal{T}_3(\mu_n,Q_n,\mu^{(x,a)}_n)(x,a).
\end{cases}
\end{align}

\subsection{Idealized Three-timescale Algorithm}\label{sec:idealalgo}

Finally, we remove the stochastic approximation step by using expectations of the random terms in the synchronous algorithm. More precisely, in the first item in Section \ref{sec:syncalgo},
$\widecheck{\mathcal 
 {P}}_{X^{(x,a)}_n,\tilde\pi^{(x,a)}(X^{(x,a)}_n)}(\mu^{(x,a)})$ is replaced by 
 $\mathcal{\tilde P}^{(x,a)}_3(\mu_n,Q_n,\mu_n^{(x,a)})$, in the second item 
$ \widecheck{\mathcal 
 {P}}_{X_n,\softmin Q_n(X_n)}(\mu)$
 is replaced by $\mathcal{P}_3(\mu_n,Q_n,\mu_n^{(x,a)})$, and in the third item, 
$\widecheck{\mathcal{T}}_{\mu_n,Q_n,\mu^{(x,a)}_n}(x,a)$ is replaced by
$\mathcal{T}_3(\mu_n,Q_n,\mu_n^{(x,a)})(x,a)$.
That leads to the system of iterations:
\begin{align}
    \begin{cases}
    \label{DEs}
    \mu^{(x,a)}_{n+1} = \mu^{(x,a)}_n + \rho_n^{\tilde\mu}\mathcal{\tilde P}^{(x,a)}_3(\mu_n,Q_n,\mu_n^{(x,a)})\\
        \mu_{n+1} = \mu_n + \rho_n^{\mu}\mathcal{P}_3(\mu_n, Q_n,\mu^{(x,a)}_n)\\
        Q_{n+1}(x,a) = Q_n(x,a) + \rho^Q_n\mathcal{T}_3(\mu_n, Q_n,\mu^{(x,a)}_n)
        \end{cases}
\end{align}
where
\begin{align}
    \begin{cases}
    \label{def:P3T3old}
\tilde{\mathcal{P}}^{(x,a)}_3(\mu,Q,\tilde\mu)(y) 
    &:= \sum_{x'\neq x}\tilde\mu(x')p(y|x',{\softmin}_\phi Q(x'),\mu,\tilde\mu) +\tilde\mu(x)p(y|x,a,\mu,\tilde\mu) - \tilde\mu(y)
    \\
    \mathcal{P}_3(\mu, Q,\tilde\mu)(y)
    &:=\sum_{x'}\mu(x')p(y|x',{\softmin}_\phi Q(x'),\mu,\tilde\mu)-\mu(y)
    \\
\mathcal{T}_3(\mu,Q,\tilde\mu)(x,a) &:= f(x,a,\mu,\tilde\mu) + \gamma\sum_{x'}p(x'|x,a,\mu.\tilde\mu)\min_{a'}Q(x',a')-Q(x,a),
    \end{cases}
\end{align}
Note that the first two equations in \eqref{def:P3T3old} can be rewritten
\begin{align}
    \begin{cases}
    \label{def:P3T3bis}
\tilde{\mathcal{P}}^{(x,a)}_3(\mu,Q,\tilde\mu)(y) &=
    \tilde\mu \mathrm{\tilde P}_{(x,a)}^{\softmin_\phi Q, \mu,\tilde\mu}(y)-\tilde\mu(y),
    \\
    \mathcal{P}_3(\mu, Q,\tilde\mu)(y)
    &=\mu\mathrm{ P}^{\softmin_\phi Q, \mu,\tilde\mu}(y)-\mu(y),
    \end{cases}
\end{align}
where $\mathrm{P}^{\pi,\mu,\tilde\mu} $ and $\mathrm{\tilde P}_{(x,a)}^{\pi, \mu,\tilde\mu}$
are defined by
\begin{align}
    \begin{cases}
    \label{def:P3T3}
\mu'\mathrm{P}^{\pi,\mu,\tilde\mu}(y)=\sum_{x'\in \mathcal{X}} \mu'(x') \sum_{a\in \mathcal{A}} \pi(a|x') p(y|x',a,\mu,\tilde\mu), \quad y \in \mathcal{X}
    \\
 \mu' \mathrm{\tilde P}_{(x,a)}^{\pi, \mu,\tilde\mu}(y)=
 \sum_{x'\neq x}\mu'(x')p(y|x',\pi,\mu,\tilde\mu) +\mu'(x)p(y|x,a,\mu,\tilde\mu)
 , \quad y \in \mathcal{X}.
    \end{cases}
\end{align}

Now, following \cite{Borkar97} in the two-timescale case, we consider the following three-timescale ODE system which tracks the system \eqref{DEs}:
\begin{align}
\label{ODEs}
    \begin{cases}
    \dot{\mu}_t = \mathcal{P}_3(\mu_t,Q_t,\mu^{(x,a)}_t)\\
    \dot{Q}_t(x,a) = \frac{1}{\epsilon}\mathcal{T}_3(\mu_t,Q_t,\mu^{(x,a)}_t)\\
    \dot{\mu}^{(x,a)}_t = \frac{1}{\epsilon\tilde{\epsilon}}\tilde{\mathcal{P}}^{(x,a)}_3(\mu_t,Q_t,\mu^{(x,a)}_t),
\end{cases}
\end{align} 
where the small parameters $\epsilon$ and $\tilde{\epsilon}$ represents the ratios $\rho_n^\mu/\rho_n^Q$ and $\rho_n^Q/\rho_n^{\tilde\mu}$ respectively. 
This is what we call the {\bf idealized algorithm}, because it cannot be implemented directly in general since the expectations cannot be computed exactly.

\section{Convergence: Three-timescale Approach  }\label{sec:convergence}

As announced in Section \ref{sec:algos}, we now proceed with the proof of convergence starting with convergence of the ideal deterministic algorithm by using a three-timescale generalization of two-timescale results in \cite{Borkar97}, and then successively restoring the stochastic approximation  and the asynchronous features by controlling the corresponding errors.

\subsection{Convergence of the Idealized Three-timescale Algorithm}
In this section, we  establish the convergence of the idealized three-timescale algorithm defined in \eqref{DEs}. In order to do so, we will introduce several assumptions, the first of which is the following regarding regularity of the functions $f$ and $p$.

\begin{assumption}\label{mfcglip}
    The cost function $f$ is bounded and is Lipschitz with respect to $\mu$ and $\mu'$, with Lipschitz constants denoted by $L_{f}^{glob}$ and $L_{f}^{loc}$ respectively, when using the $L^1$ norm.
    The transition kernel $p(\cdot|\cdot,\cdot,\mu,\mu')$ is also Lipschitz with respect to $\mu$ and $\mu'$, with Lipschitz constants denoted by $L_{p}^{glob}$ and $L_{p}^{loc}$ respectively, when using the $L^1$ norm. In other words, for every $x,a,\mu,\mu', \tilde\mu,\tilde\mu'$, we have:
    \begin{align*}
        |f(x,a,\mu,\tilde\mu) - f(x,a,\mu',\tilde\mu)| 
        &\le L_{f}^{glob} \|\mu - \mu'\|_1 = L_{f}^{glob} \sum_x |\mu(x) - \mu'(x)|,
        \\ 
        |f(x,a,\mu,\tilde\mu) - f(x,a,\mu,\tilde\mu')| 
        &\le L_{f}^{loc} \|\tilde\mu - \tilde\mu'\|_1 = L_{f}^{loc} \sum_x |\tilde\mu(x) - \tilde\mu'(x)|,
        \\ 
        \|p(\cdot|x,a,\mu,\tilde\mu)-p(\cdot|x,a,\mu',\tilde\mu)\|_1 
        &\le L_{p}^{glob} \|\mu - \mu'\|_1 = L_{p}^{glob} \sum_x |\mu(x) - \mu'(x)|,
        \\
        \|p(\cdot|x,a,\mu,\tilde\mu)-p(\cdot|x,a,\mu,\tilde\mu')\|_1 
        &\le L_{p}^{loc} \|\tilde\mu - \tilde\mu'\|_1 = L_{p}^{loc} \sum_x |\tilde\mu(x) - \tilde\mu'(x)|.
    \end{align*}
We will denote $L_p^{max}:=\max(L_p^{glob},L_p^{loc})$. 
\end{assumption}
Under Assumption \ref{mfcglip}, %
we derive in the following propositions, the Lipschitz properties with respect to $\mu,Q,\tilde\mu$ of the functions $\mathcal{P}_3,\tilde{\mathcal{P}}^{(x,a)}_3$, and $\mathcal{T}_3$
defined in \eqref{def:P3T3old}.

First, we introduce the notations:
 \begin{align}
            c_{min}^\phi:= &\min_{x,x',Q,\mu,\tilde\mu} \mathrm{P}^{\softmin_\phi Q,\mu,\tilde\mu}(x',x) = \min_{x,x',Q,\mu,\tilde\mu}\sum_a {\softmin}_{\phi} Q(x',a)p(x|x',a,\mu,\tilde\mu),
            \\
        c_{min} :=& \min_{x,x',a,\mu,\tilde\mu} p(x|x',a,\mu,\tilde\mu)\sum_a {\softmin}_{\phi} Q(x',a) =\min_{x,x',a,\mu,\tilde\mu} p(x|x',a,\mu,\tilde\mu). \label{def:cmin}
\end{align}
Note that $c_{min} \le c_{min}^\phi$. 
\begin{proposition}\label{PLipschitz}
If Assumption~\ref{mfcglip} holds, then
    the function $(\mu,Q,\tilde\mu)\mapsto\mathcal{P}_3(\mu,Q,\tilde\mu)$ is Lipschitz and more precisely: %
\begin{align}\label{LipP3Q}
        \|\mathcal{P}_3(\mu,Q,\tilde\mu) - \mathcal{P}_3(\mu, Q,\tilde\mu')\|_{\infty}
\leq L^{loc}_p\|\tilde\mu-\tilde\mu'\|_1,
\end{align}
\begin{align}\label{LipP3Q}
        \|\mathcal{P}_3(\mu,Q,\tilde\mu) - \mathcal{P}_3(\mu, Q',\tilde\mu)\|_{\infty}
\leq \phi|\mathcal{A}|\|Q-Q'\|_\infty,
\end{align}
     and
     \begin{align}\label{LipP3mu}
        \|\mathcal{P}_3(\mu,Q,\tilde\mu) - \mathcal{P}_3(\mu',Q,\tilde\mu)\|_{\infty}
         \leq (L_p^{glob} + 2-|\mathcal{X}|c_{min}^\phi)\|\mu-\mu'\|_1.
    \end{align}

    \end{proposition}
\begin{proof}
We have:
\begin{align*}
        \|\mathcal{P}_3(\mu,Q,\tilde\mu) - \mathcal{P}_3(\mu,Q,\tilde\mu')\|_{\infty} &\leq \|\mu\mathrm{P}^{\softmin_\phi Q, \mu, \tilde\mu} - \mu \mathrm{P}^{\softmin_\phi Q, \mu, \tilde\mu'}\|_\infty \\
        &\leq \|\mu\mathrm{P}^{\softmin_\phi Q, \mu, \tilde\mu} - \mu\mathrm{P}^{\softmin_\phi Q, \mu, \tilde\mu'})\|_1\\
        &\leq L_p^{loc} \|\tilde\mu-\tilde\mu'\|_1 .
\end{align*}
Moreover, 
\begin{align*}
        & \|\mathcal{P}_3(\mu,Q,\tilde\mu) - \mathcal{P}_3(\mu,Q',\tilde\mu)\|_{\infty} 
        \\
        & = \|\mu \mathrm{P}^{\softmin_\phi Q, \mu, \tilde\mu} - \mu \mathrm{P}^{\softmin_\phi Q', \mu, \tilde\mu}\|_\infty
        \\
        &\leq\sum_{x}\left|\sum_{x'} \mu(x')\sum_a(({\softmin}_\phi Q(x'))(a)p(x|x',a,\mu,\tilde\mu)-({\softmin}_{\phi} Q'(x'))(a)p(x|x',a,\mu,\tilde\mu))\right|\\
        &\leq\sum_{x'} \mu(x')\left|\sum_a\sum_{x}(({\softmin}_\phi Q(x'))(a)p(x|x',a,\mu,\tilde\mu)-({\softmin}_{\phi} Q'(x'))(a)p(x|x',a,\mu,\tilde\mu))\right|
        \\
        &\leq \sum_{x'} \mu(x') \sqrt{\sum_a 1^2}\|{\softmin}_\phi Q(x')-{\softmin}_{\phi} Q'(x')\|_2\\
        &\leq\sum_{x'} \mu(x')|\mathcal{A}|^{\frac{1}{2}}\phi\|Q(x')-Q'(x')\|_2\\
        &\leq \phi|\mathcal{A}|\|Q-Q'\|_\infty.
\end{align*}    
Last, for the Lipschitz continuity with respect to $\mu$, we first make the following remark: we can find $p(i,j)>\epsilon$, so that $1-N\epsilon>0$, and then define $q(i,j)=(p(i,j)-\epsilon)/(1-N\epsilon)$. Thus, $P = (1-N\epsilon)Q+\epsilon J$, where $J$ is the $N\times N$ matrix with all entries 1. We use this fact to calculate the total variation. Hence we have: 
\begin{align*}
        \|\mathcal{P}_3(\mu,Q,\tilde\mu) - \mathcal{P}_3(\mu',Q,\tilde\mu)\|_{\infty} &\leq \|\mu\mathrm{P}^{\softmin_\phi Q, \mu, \tilde\mu} - \mu' \mathrm{P}^{\softmin_\phi Q, \mu', \tilde\mu}\|_\infty + \|\mu-\mu'\|_\infty \\
        &\leq \|\mu\mathrm{P}^{\softmin_\phi Q, \mu, \tilde\mu} - \mu'\mathrm{P}^{\softmin_\phi Q, \mu, \tilde\mu})\|_1 
        \\
        & \qquad + \|\mu'\mathrm{P}^{\softmin_\phi Q, \mu, \tilde\mu} - \mu'\mathrm{P}^{\softmin_\phi Q, \mu', \tilde\mu}\|_1+\|\mu-\mu'\|_1\\
        &\leq(1-|\mathcal{X}|c_{min}^\phi)\|\mu-\mu'\|_1 + L_p^{glob}\|\mu-\mu'\|_1+\|\mu-\mu'\|_1\\
        &\leq (L_p^{glob} + 2-|\mathcal{X}|c_{min}^\phi)\|\mu-\mu'\|_1.
\end{align*}
\end{proof}

\begin{proposition}\label{TPLipschitz}
If Assumption~\ref{mfcglip} holds, then for every $(x,a)$, 
    the function $(\mu,Q,\tilde\mu) \mapsto \mathcal{\tilde P}^{(x,a)}_3(\mu,Q,\tilde\mu)$ is Lipschitz and more precisely: %
        \begin{align}\label{LipTP3mu}
        \|\mathcal{\tilde P}^{(x,a)}_3(\mu,Q,\tilde\mu) - \mathcal{\tilde P}^{(x,a)}_3(\mu',Q,\tilde\mu)\|_{\infty}
    \leq L^{glob}_p \|\mu-\mu'\|_1,
    \end{align}
    \begin{align}\label{LipTP3Q}
        \|\mathcal{\tilde P}^{(x,a)}_3(\mu,Q,\tilde\mu) - \mathcal{\tilde P}^{(x,a)}_3(\mu,Q',\tilde\mu)\|_{\infty}
    \leq \phi|\mathcal{A}| \|Q-Q'\|_\infty,
    \end{align}
and
     \begin{align}\label{LipTP3tildemu}
        \|\mathcal{\tilde P}^{(x,a)}_3(\mu,Q,\tilde\mu) - \mathcal{\tilde P}^{(x,a)}_3(\mu,Q,\tilde\mu')\|_{\infty}
         \leq (L^{loc}_p + 2-|\mathcal{X}|c_{min}) \|\tilde\mu-\tilde\mu'\|_1.
    \end{align}
    \end{proposition}
\begin{proof}
Let us fix $(x,a)$. First, we have
\begin{align*}
        \|\mathcal{\tilde P}^{(x,a)}_3(\mu,Q,\tilde\mu) - \mathcal{\tilde P}^{(x,a)}_3(\mu',Q,\tilde\mu)\|_{\infty} 
        &\leq \|\tilde\mu\mathrm{\tilde P}_{(x,a)}^{\softmin_\phi Q, \mu, \tilde\mu} - \tilde\mu\mathrm{\tilde P}_{(x,a)}^{\softmin_\phi Q, \mu', \tilde\mu}\|_\infty 
        \\
        &\leq \|\sum_{x'\neq x}\tilde\mu(x')p(\cdot|x',{\softmin}_\phi Q(x'),\mu,\tilde\mu) +\tilde\mu(x)p(\cdot|x,a,\mu,\tilde\mu) 
        \\
        &\qquad - \sum_{x'\neq x}\tilde\mu(x')p(\cdot|x',{\softmin}_\phi Q(x'),\mu',\tilde\mu) +\tilde\mu(x)p(\cdot|x,a,\mu',\tilde\mu)\|_\infty 
        \\
        &\leq \sum_{x'}\mu(x')\|\max_{x',a'} p(\cdot|x',a',\mu,\tilde\mu) - p(\cdot|x',a',\mu',\tilde\mu)\|_\infty 
        \\
        &\leq L^{glob}_p \|\mu-\mu'\|_1.
\end{align*}    
Second, we have:
\begin{align*}
        \|\mathcal{\tilde P}^{(x,a)}_3(\mu,Q,\tilde\mu) - \mathcal{\tilde P}^{(x,a)}_3(\mu,Q',\tilde\mu)\|_{\infty} 
        &= \|\mu\mathrm{\tilde P}_{(x,a)}^{\softmin_\phi Q, \mu, \tilde\mu} - \mu\mathrm{\tilde P}_{(x,a)}^{\softmin_\phi Q', \mu, \tilde\mu}\|_\infty
        \\
        &\leq \|\sum_{x'\neq x}\tilde\mu(x')p(\cdot|x',{\softmin}_\phi Q(x'),\mu,\tilde\mu) +\tilde\mu(x)p(\cdot|x,a,\mu,\tilde\mu) 
        \\
        &\qquad - \sum_{x'\neq x}\tilde\mu(x')p(\cdot|x',{\softmin}_\phi Q'(x'),\mu,\tilde\mu) +\tilde\mu(x)p(\cdot|x,a,\mu,\tilde\mu)\|_\infty 
        \\
        &\leq\|\mu\mathrm{P}^{\softmin_\phi Q, \mu, \tilde\mu} - \mu\mathrm{P}^{\softmin_\phi Q', \mu, \tilde\mu}\|_{1}\\
        &\leq \phi|\mathcal{A}| \|Q-Q'\|_\infty. 
\end{align*}    
 Last, 
\begin{align*}
        \|\mathcal{\tilde P}^{(x,a)}_3(\mu,Q,\tilde\mu) - \mathcal{\tilde P}^{(x,a)}_3(\mu,Q,\tilde\mu')\|_{\infty} 
        &\leq\|\tilde\mu\mathrm{\tilde P}_{(x,a)}^{\softmin_\phi Q, \mu, \tilde\mu} - \tilde\mu'\mathrm{\tilde P}_{(x,a)}^{\softmin_\phi Q, \mu, \tilde\mu'}\|_\infty + \|\tilde\mu-\tilde\mu'\|_\infty \\
        &\leq\|\tilde\mu\mathrm{\tilde P}_{(x,a)}^{\softmin_\phi Q, \mu, \tilde\mu} - \tilde\mu'\mathrm{\tilde P}_{(x,a)}^{\softmin_\phi Q, \mu, \tilde\mu'}\|_{1} + \|\tilde\mu-\tilde\mu'\|_{1}\\
        &\leq\|\tilde\mu\mathrm{P}^{\softmin_\phi Q, \mu, \tilde\mu} - \tilde\mu'\mathrm{P}^{\softmin_\phi Q, \mu, \tilde\mu'}\|_{1} + \|\tilde\mu-\tilde\mu'\|_{1}\\
        &\leq (L^{loc}_p + 2-|\mathcal{X}|c_{min})\|\tilde\mu-\tilde\mu'\|_1.
\end{align*}
\end{proof}

\begin{proposition}\label{TLipschitz}
    If Assumption~\ref{mfcglip} holds,
    then the function $(\mu,Q,\tilde\mu) \mapsto \mathcal{T}_3(\mu,Q,\tilde\mu)$ is Lipschitz. %
\end{proposition}
\begin{proof}
We recall that the Bellman operator $\mathcal{B}_{\mu,\tilde\mu}$ has been introduced in~\eqref{eq:BQ}. 
We note that:
    \begin{align*}
        \|\mathcal{T}_3(\mu,Q,\tilde\mu) - \mathcal{T}_3(\mu, Q',\tilde\mu)\|_{\infty} &\leq \|\mathcal{B}_{\mu,\tilde\mu} Q - \mathcal{B}_{\mu,\tilde\mu} Q'\|_\infty + \|Q-Q'\|_\infty\\
        &\leq \gamma \|\sum_{x'} p(x'|\cdot,\cdot,\mu,\tilde\mu)(\min_{a'} Q(x',a')-\min_{a'} Q'(x',a'))\|_\infty+\|Q-Q'\|_\infty\\
        &\leq \gamma \|Q-Q'\|_\infty + \|Q-Q'\|_\infty\\
        &\leq (\gamma + 1)\|Q-Q'\|_\infty,
    \end{align*}
    where we used the fact that 
    $$
        \|\sum_{x'} p(x'|\cdot,\cdot,\mu,\tilde\mu)(\min_{a'} Q(x',a')-\min_{a'} Q'(x',a'))\|_\infty \le \sup_{x,a}  \sum_{x'} p(x'|x,a,\mu,\tilde\mu)|(\min_{a'} Q(x',a')-\min_{a'} Q'(x',a'))|,
    $$ 
    with $p(x'|x,a,\mu,\tilde\mu) \le 1$ and $|\min_{a'} Q(x',a')-\min_{a'} Q'(x',a')| \le \|Q- Q'\|_\infty$.  Moreover, 
    \begin{align*}
        \|\mathcal{T}_3(\mu,Q,\tilde\mu) - \mathcal{T}_2(\mu,Q,\tilde\mu')\|_{\infty} &\leq \|f(\cdot,\cdot,\mu,\tilde\mu) - f(\cdot,\cdot,\mu,\tilde\mu')\|_\infty + \|\gamma\sum_{x'}|p(x'|\cdot,\cdot,\mu,\tilde\mu)-p(x'|\cdot,\cdot,\mu,\tilde\mu')||\min_{a'} Q(x',a')|\|_\infty\\
        &\leq (L_f^{glob} + \gamma L_p^{glob}\|Q\|_\infty)\|\tilde\mu-\tilde\mu'\|_1\\
        &\leq |\mathcal{X}|(L_f^{glob} + \gamma L_p^{glob}\|Q\|_\infty)\|\tilde\mu-\tilde\mu'\|_\infty.
    \end{align*}
    
    Likewise, we can show:
    \begin{align*}
        \|\mathcal{T}_3(\mu,Q,\tilde\mu) - \mathcal{T}_2(\mu',Q,\tilde\mu)\|_{\infty} 
        &\leq |\mathcal{X}|(L_f^{loc} + \gamma L_p^{loc}\|Q\|_\infty)\|\mu-\mu'\|_\infty.
    \end{align*}
\end{proof}

Now that we have these Lipschitz properties, we show that the ODE system (\ref{DEs}) has a unique global asymptotically stable equilibrium (GASE).
\begin{assumption}\label{mfcglp}
    Let $L_p^{loc} < \frac{1}{2}|\mathcal{X}|c_{min}$ and $L_p^{glob} < \frac{1}{2}|\mathcal{X}|c_{min}$ .
\end{assumption}
We first show that the fast scale variable, which corresponds to the local distribution,  converges to a GASE.
\begin{proposition}\label{emu'_MFCG}
    Suppose Assumptions \ref{mfcglip} and \ref{mfcglp} hold. For any given $Q$ and $\mu$, for every $(x,a)$, $\dot{\tilde\mu}_t = \tilde{\mathcal{P}}_3^{(x,a)}(\mu,Q,\tilde\mu_t)$ has a unique GASE, that we will denote by $\tilde\mu^{*\phi, (x,a)}_{Q,\mu}$. Moreover, $\tilde\mu^{*\phi, (x,a)}:\mathbbm{R}^{|\mathcal{X}|\times|\mathcal{A}|} \times \Delta^{|\mathcal{X}|}\to\Delta^{|\mathcal{X}|}$, $(\mu,Q) \mapsto \tilde\mu^{*\phi, (x,a)}_{\mu,Q}$ is Lipschitz.
\end{proposition}
\begin{proof}
        Let us fix $\mu$ and $Q$. By Proposition~\ref{TPLipschitz}, we have $ \|\tilde\mu\mathrm{\tilde P}^{\softmin_\phi Q, \mu, \tilde\mu}_{(x,a)}-\tilde\mu' \mathrm{\tilde P}^{\softmin_\phi Q, \mu, \tilde\mu'}_{(x,a)}\|_1 \leq (L^{loc}_p + 1-|\mathcal{X}|c_{min}) \|\tilde\mu-\tilde\mu'\|_1$. Since $(L^{loc}_p + 1-|\mathcal{X}|c_{min}) < 1$ by Assumption~\ref{mfcglp},  $\tilde\mu \mapsto  \tilde\mu \tilde{\mathrm{P}}^{\softmin_\phi Q, \mu,  \tilde\mu}_{(x,a)}$ is a strict contraction. As a result, by the contraction mapping theorem \cite{SELL197342}, 
        there is a unique GASE and furthermore, by \cite[Theorem 3.1]{563625}, it is the limit of $\tilde\mu_t$ denoted by $\tilde\mu^{*\phi, (x,a)}_{\mu,Q}$. Then,
    \begin{align*}
        \|\tilde\mu^{*\phi, (x,a)}_{\mu,Q} - \tilde\mu^{*\phi, (x,a)}_{\mu',Q'}\|_1 
        &\leq \|\tilde\mu^{*\phi, (x,a)}_{\mu,Q} \tilde{\mathrm{P}}^{\softmin_\phi Q, \mu, \tilde\mu^{*\phi, (x,a)}_{\mu,Q}}_{(x,a)} - \tilde\mu^{*\phi, (x,a)}_{\mu',Q'}\tilde{\mathrm{P}}^{\softmin_\phi Q', \mu', \tilde\mu^{*\phi, (x,a)}_{\mu',Q'}}_{(x,a)}\|_1
        \\
        &\leq \|\tilde\mu^{*\phi, (x,a)}_{\mu,Q} \tilde{\mathrm{P}}^{\softmin_\phi Q, \mu, \tilde\mu^{*\phi, (x,a)}_{\mu,Q}}_{(x,a)} - \tilde\mu^{*\phi, (x,a)}_{\mu',Q'}\tilde{\mathrm{P}}^{\softmin_\phi Q, \mu, \tilde\mu^{*\phi, (x,a)}_{\mu',Q'}}_{(x,a)}\|_1
        \\
        &\qquad + \|\tilde\mu^{*\phi, (x,a)}_{\mu',Q'}\tilde{\mathrm{P}}^{\softmin_\phi Q, \mu, \tilde\mu^{*\phi, (x,a)}_{\mu',Q'}}_{(x,a)} - \tilde\mu^{*\phi, (x,a)}_{\mu',Q'}\tilde{\mathrm{P}}^{\softmin_\phi Q, \mu', \tilde\mu^{*\phi, (x,a)}_{\mu',Q'}}_{(x,a)}\|_1
        \\
        &\qquad + \|\tilde\mu^{*\phi, (x,a)}_{\mu',Q'}\tilde{\mathrm{P}}^{\softmin_\phi Q, \mu', \tilde\mu^{*\phi, (x,a)}_{\mu',Q'}}_{(x,a)} - \tilde\mu^{*\phi, (x,a)}_{\mu',Q'}\tilde{\mathrm{P}}^{\softmin_\phi Q', \mu', \tilde\mu^{*\phi, (x,a)}_{\mu',Q'}}_{(x,a)}\|_1
        \\
        &\le (L^{loc}_p + 1 -|\mathcal{X}|c_{min})\|\tilde\mu^{*\phi, (x,a)}_{\mu,Q} - \tilde\mu^{*\phi, (x,a)}_{\mu',Q'}\|_1
        \\
        &\qquad + L_p^{glob} \|\mu - \mu'\|_1
        + \phi|\mathcal{A}| \|Q-Q'\|_\infty.
    \end{align*}
    Furthermore, $(L^{loc}_p + 1 -|\mathcal{X}|c_{min})<1$ by Assumption~\ref{mfcglp} and  as a result, $\tilde\mu^{*\phi, (x,a)}_{\mu,Q}$ is uniformly Lipschitz with respect to $\mu$ and $Q$. 
\end{proof}

\begin{assumption}\label{SecondGASE}
For any fixed $\mu$, the second ODE in the system \eqref{ODEs} along the GASE of the first ODEs, that is
$\dot{Q}_t = \mathcal{T}_3(\mu,Q_t,{\tilde\mu^{*\phi,(x,a)}_{Q_t,\mu}})$, has a unique GASE. We will denote it by $Q^{*\phi}_\mu$.
\end{assumption}

\begin{assumption}\label{ThirdGASE}
The third ODE in system \eqref{ODEs} along the GASE of the second ODE, that is $\dot{\mu}_t = \mathcal{P}_3(\mu_t, Q^{*\phi}_{\mu_t},{\tilde\mu^{*\phi,(x,a)}_{Q^{*\phi}_{\mu_t},\mu_t}})$, has a unique GASE. We will denote it by $\mu^{*\phi}$.
\end{assumption}

In Appendix \ref{gasecontraction}, we show how to derive the existence and uniqueness of these two GASEs using a contraction argument. However, this argument requires an upper bound on the choice of the parameter $\phi$ which limits its use in our context of deriving the convergence of our algorithm to a solution of the MFCG problem. Therefore, in what follows we work under Assumption \ref{SecondGASE} and Assumption \ref{ThirdGASE}   which could alternatively be verified by means of Lyapunov functions for instance, as we do in the example presented in Section \ref{sec:example}.

We can now introduce our first theorem, which guarantees the convergence of the idealized three-timescale algorithm. It relies on the following assumptions about the update rates:
\begin{assumption}\label{squaresumlearningrate_mfcg} 
    The learning rates $\rho^{\tilde{\mu}}_n$, $\rho^Q_n$ and $\rho^{\mu}_n$ are sequences of positive real numbers satisfying
    \begin{align*}
    \sum_n  \rho^{\tilde{\mu}}_n=  \sum_n \rho^Q_n = \sum_n \rho^{\mu}_n = \infty, \quad \sum_n |\rho^{\tilde{\mu}}_n|^2+|\rho^Q_n|^2 + |\rho^{\mu}_n|^2 <\infty,
        \quad \rho^\mu_n/\rho^Q_n \xrightarrow[n\to+\infty]{} 0, \quad \rho^Q_n /\rho^{\tilde{\mu}}_n\xrightarrow[n\to+\infty]{} 0.
    \end{align*}
\end{assumption}

\begin{theorem}\label{thMFCGcvg}
    Suppose Assumptions \ref{mfcglip}, \ref{mfcglp}, \ref{SecondGASE}, \ref{ThirdGASE} and \ref{squaresumlearningrate_mfcg} hold. Then, $(\mu_n, Q_n,\tilde\mu_n^{(x,a)})$ defined in \eqref{DEs} converge to $(\mu^{*\phi},Q^{*\phi}_{\mu^{*\phi}},\tilde\mu^{*\phi,(x,a)}_{Q^{*\phi},\mu^{*\phi}})$ as $n \to\infty$.
\end{theorem}
\begin{proof}
    With Assumptions \ref{mfcglip}, \ref{mfcglp}, \ref{SecondGASE}, \ref{ThirdGASE}, \ref{squaresumlearningrate_mfcg},  and the results of Propositions \ref{PLipschitz}, \ref{TPLipschitz}, \ref{TLipschitz}, and \ref{emu'_MFCG}, the assumptions of \cite[Theorem 1.1]{Borkar97} are satisfied, and therefore  guarantees the convergence in the statement.
\end{proof}

\begin{assumption}\label{MFCGuniqueness}
In what follows, we assume that there exists a unique solution to our MFCG problem in the sense of Definition \ref{mfcg_definition}, and we further assume that this solution is a pure policy.
\end{assumption}
We will need the following result:
\begin{lemma}\label{musystem}
For any given pure policy $\alpha$, the following system for $\mu$ and $\tilde\mu^{(x,a)}$
\begin{align}\label{mus}
    \begin{cases}
        \tilde\mu^{(x,a)}=\tilde\mu^{(x,a)} \tilde{\mathrm{P}}^{\alpha, \mu, \tilde\mu^{(x,a)}}_{(x,a)}\\
        \mu=\mu\mathrm{P}^{\alpha, \mu, \tilde\mu^{(x,\alpha(x))}},
    \end{cases}
\end{align}
where $\mathrm{P}^{\pi,\mu,\tilde\mu}$ and $\tilde{\mathrm{P}}_{(x,a)}^{\pi, \mu, \tilde\mu}$ are defined in
\eqref{def:P3T3}, has a unique solution. Furthermore, $\tilde\mu^{(x,\alpha(x))}=\mu$, which is independent of $x$.
\end{lemma}
\begin{proof}
See Appendix \ref{prooflemma}.
\end{proof}

Next, we show that the limit point given by the algorithm, that is $(\mu^{*\phi},Q^{*\phi}_{\mu^{*\phi}},\tilde\mu^{*\phi,(x,a)}_{Q^{*\phi},\mu^{*\phi}})$, provides an approximation of this MFCG solution. 

\begin{definition}\label{MFCGS} 
Given $(\mu^{*\phi},Q^{*\phi}_{\mu^{*\phi}},\tilde\mu^{*\phi,(x,a)}_{Q^{*\phi},\mu^{*\phi}})$, we define the policy $\alpha^*$, the distributions $\mu^*$ and $\tilde\mu^{*,(x,a)}$ for every $(x,a)$, and the state-action value function $Q^*$ as:
    \begin{enumerate}
         \item $\alpha^*(x) = \argmin_{a'} Q^{*\phi}_{\mu^{*\phi}}(x,a')$
         \item  $\mu^*$ and $\tilde\mu^{*,(x,a)}$ are the fixed points of the following equations:
         \[
         \tilde\mu^{*,(x,a)}=\tilde\mu^{*,(x,a)} \tilde{\mathrm{P}}^{\alpha^*, \mu^*, \tilde\mu^{*,(x,a)}}_{(x,a)},\quad \mu^*=\mu^*\mathrm{P}^{\alpha^*, \mu^*, \tilde\mu^{*,(x,\alpha^*(x))}},
         \]
           where we use Lemma \ref{musystem} with $\alpha=\alpha^*$.
        \item For these $\mu^*$ and $\tilde\mu^{*,(x,a)}$, $Q^*$ is the solution to the Bellman equation~\eqref{MKV-Bellman}.
\end{enumerate}
\end{definition}

\begin{remark}\label{mu=mu}
By choosing $a=\alpha^*(x)$ in item 2, by Lemma \ref{musystem}, we have $\tilde\mu^{*,(x,\alpha^*(x))}=\mu^*$. This is expected as in the limiting equilibrium local and global distributions are indistinguishable.
\end{remark}

The following result shows that $(\mu^{*\phi},Q^{*\phi}_{\mu^{*\phi}},\tilde\mu^{*\phi,(x,a)}_{Q^{*\phi},\mu^{*\phi}})$ given by Theorem \ref{thMFCGcvg} are close to $(\mu^*,Q^*,\tilde\mu^{*,(x,a)})$.
\begin{theorem} \label{mfcg_main}
Suppose Assumptions \ref{mfcglip}, \ref{mfcglp}, \ref{SecondGASE}, \ref{ThirdGASE} and \ref{squaresumlearningrate_mfcg} hold.
Let $(\mu^*, Q^*,\tilde\mu^{*,(x,a)})$ given by Definition \ref{MFCGS}.
    Let $\delta(\phi)$ be the action gap defined as $\delta(\phi) = \min_{x\in\mathcal{X}}(\min_{a\notin \argmin_a Q^{*\phi}_{\mu^{*\phi}}}Q^{*\phi}_{\mu^{*\phi}}(x,a) -\min_a Q^{*\phi}_{\mu^{*\phi}}(x,a)) > 0$, and $\delta(\phi) =\infty$ if $Q^{*\phi}_{\mu^{*\phi}}(x)$ is constant with respect to $a$ for each $x$. 
    Then, 
    \begin{align}\label{mfcg_errorphi}
    \begin{cases}
        \|\tilde\mu^{*\phi,(x,a)}_{Q^{*\phi},\mu^{*\phi}} - \tilde\mu^{*,(x,a)}\|_1 + \|\mu^{*\phi}-\mu^*\|_1\leq \frac{4|\mathcal{A}|^{\frac{3}{2}}\exp(-\phi\delta(\phi))}{|\mathcal{X}|c_{min}-2L^{max}_p}
        \\
        \|Q^{*\phi}_{\mu^{*\phi}} - Q^*\|_\infty \leq \frac{(L_f + \frac{\gamma}{1-\gamma} L^{max}_p\|f\|_\infty)4|\mathcal{A}|^{\frac{3}{2}}\exp(-\phi\delta(\phi))}{(1-\gamma)(|\mathcal{X}|c_{min}-2L^{max}_p)}.
    \end{cases}
    \end{align}
\end{theorem}

\begin{proof}
    Using a similar argument as in Proposition \ref{PLipschitz} for the first term of the second inequality below, we have:
\begin{align*}
     \|\tilde\mu^{*\phi,(x,a)}_{Q^{*\phi},\mu^{*\phi}} - \tilde\mu^{*,(x,a)}\|_1 &= \|\tilde\mu^{*\phi,(x,a)}_{Q^{*\phi},\mu^{*\phi}}\mathrm{\tilde P}_{(x,a)}^{\softmin_\phi Q^{*\phi}_{\mu^{*\phi}},\mu^{*\phi},\tilde\mu^{*\phi,(x,a)}_{Q^{*\phi},\mu^{*\phi}}} - \tilde\mu^{*,(x,a)}\mathrm{\tilde P}_{(x,a)}^{\alpha^*,\mu^*,\tilde\mu^{*,(x,a)}}\|_1 \\
    &\leq\|\tilde\mu^{*\phi,(x,a)}_{Q^{*\phi},\mu^{*\phi}}\mathrm{\tilde P}_{(x,a)}^{\softmin_\phi Q^{*\phi}_{\mu^{*\phi}},\mu^{*\phi},\tilde\mu^{*\phi,(x,a)}_{Q^{*\phi},\mu^{*\phi}}} - \tilde\mu^{*\phi,(x,a)}_{Q^{*\phi},\mu^{*\phi}}\mathrm{\tilde P}_{(x,a)}^{\alpha^*,\mu^{*\phi},\tilde\mu^{*\phi,(x,a)}_{Q^{*\phi},\mu^{*\phi}}}\|_1\\
    &\qquad+\|\tilde\mu^{*\phi,(x,a)}_{Q^{*\phi},\mu^{*\phi}}\mathrm{\tilde P}_{(x,a)}^{\alpha^*,\mu^{*\phi},\tilde\mu^{*\phi,(x,a)}_{Q^{*\phi},\mu^{*\phi}}} - \tilde\mu^{*,(x,a)}\mathrm{\tilde P}_{(x,a)}^{\alpha^*,\mu^{*\phi},\tilde\mu^{*,(x,a)}}\|_1 \\
    &\qquad+\|\tilde\mu^{*,(x,a)}\mathrm{\tilde P}_{(x,a)}^{\alpha^*,\mu^{*\phi},\tilde\mu^{*,(x,a)}}-\tilde\mu^{*,(x,a)}\mathrm{\tilde P}_{(x,a)}^{\alpha^*,\mu^*,\tilde\mu^{*,(x,a)}}\|_1\\
    &\leq|\mathcal{A}|^{\frac{1}{2}}\max_{x\in\mathcal{X}}\|{\softmin}_\phi Q^{*\phi}_{\mu^{*\phi}}(x)-\argmin Q^{*\phi}_{\mu^{*\phi}}(x)\|_2 \\
    &\qquad+ (L^{loc}_p+1-|\mathcal{X}|c_{min})\|\tilde\mu^{*\phi,(x,a)}_{Q^{*\phi},\mu^{*\phi}} - \tilde\mu^{*,(x,a)}\|_1 + L^{glob}_p\|\mu^{*\phi}-\mu^*\|_1  \\
    &\leq
    2|\mathcal{A}|^{\frac{3}{2}}\exp(-\phi\delta(\phi))+(L^{loc}_p+1-|\mathcal{X}|c_{min})\|\tilde\mu^{*\phi,(x,a)}_{Q^{*\phi},\mu^{*\phi}} - \tilde\mu^{*,(x,a)}\|_1 + L^{glob}_p\|\mu^{*\phi}-\mu^*\|_1,
\end{align*}
where the first term in the last inequality comes from~\cite[Lemma 7]{guo2019learning}.
With a similar argument, we also have 
\begin{align*}
    \|\mu^{*\phi}-\mu^*\|_1\leq 2|\mathcal{A}|^{\frac{3}{2}}\exp(-\phi\delta(\phi))+(L^{glob}_p+1-|\mathcal{X}|c_{min})\|\mu^{*\phi}-\mu^*\|_1  + L^{loc}_p\|\tilde\mu^{*\phi,(x,\alpha^*(x))}_{Q^{*\phi},\mu^{*\phi}} - \tilde\mu^{*,(x,\alpha^*(x))}\|_1.
\end{align*}
Consequently, we obtain 
\[
(|\mathcal{X}|c_{min}-2L^{loc}_p)\|\tilde\mu^{*\phi,(x,a)}_{Q^{*\phi},\mu^{*\phi}} - \tilde\mu^{*,(x,a)}\|_1 + (|\mathcal{X}|c_{min}-2L^{glob}_p) \|\mu^{*\phi}-\mu^*\|_1\leq 4|\mathcal{A}|^{\frac{3}{2}}\exp(-\phi\delta(\phi)),
\]
which implies
\[
\|\tilde\mu^{*\phi,(x,a)}_{Q^{*\phi},\mu^{*\phi}} - \tilde\mu^{*,(x,a)}\|_1 + \|\mu^{*\phi}-\mu^*\|_1\leq \frac{4|\mathcal{A}|^{\frac{3}{2}}\exp(-\phi\delta(\phi))}{|\mathcal{X}|c_{min}-2L^{max}_p},
\]
so that we deduce the first inequality in \eqref{mfcg_errorphi}.
Then, 
\begin{align*}
    \|Q^{*\phi}_{\mu^{*\phi}} - Q^*\|_\infty &= \|\mathcal{B}_{\mu^{*\phi},\tilde\mu^{*\phi}_{Q^{*\phi},\mu^{*\phi}}} Q^{*\phi}_{\mu^{*\phi}} - \mathcal{B}_{\mu^*,\tilde\mu^*} Q^*\|_\infty \\
    &\leq \|\mathcal{B}_{\mu^{*\phi},\tilde\mu^{*\phi}_{Q^{*\phi},\mu^{*\phi}}} Q^{*\phi}_{\mu^{*\phi}} - \mathcal{B}_{\mu^{*\phi},\tilde\mu^{*\phi}_{Q^{*\phi},\mu^{*\phi}}} Q^*\|_\infty + \|\mathcal{B}_{\mu^{*\phi},\tilde\mu^{*\phi}_{Q^{*\phi},\mu^{*\phi}}} Q^* - \mathcal{B}_{\mu^{*\phi},\tilde\mu^*} Q^*\|_\infty \\
    &\qquad + \|\mathcal{B}_{\mu^{*\phi},\tilde\mu^*} Q^* - \mathcal{B}_{\mu^*,\tilde\mu^*} Q^*\|_\infty\\
    &\leq \gamma \|Q^{*\phi}_{\mu^{*\phi}} - Q^*\|_\infty + (L_f + \frac{\gamma}{1-\gamma} L^{glob}_p\|f\|_\infty)\|\mu^{*\phi}-\mu^*\|_1\\
    &\qquad+ (L_f + \frac{\gamma}{1-\gamma} L^{loc}_p\|f\|_\infty)\|\tilde\mu^{*\phi}_{Q^{*\phi},\mu^{*\phi}} - \tilde\mu^{*}\|_1,
\end{align*}
and consequently,
\[
\|Q^{*\phi}_{\mu^{*\phi}} - Q^*\|_\infty \leq \frac{(L_f + \frac{\gamma}{1-\gamma} L^{max}_p\|f\|_\infty)4|\mathcal{A}|^{\frac{3}{2}}\exp(-\phi\delta(\phi))}{(1-\gamma)(|\mathcal{X}|c_{min}-2L^{max}_p)}.
\]
\end{proof}

We are now ready for the main result of this section.
\begin{theorem} \label{th:mainMFCG}
Suppose
Assumptions~\ref{mfcglip}, \ref{mfcglp}, \ref{SecondGASE}, \ref{ThirdGASE}, \ref{squaresumlearningrate_mfcg}, and \ref{MFCGuniqueness} hold. Let $\delta = \liminf_{\phi\to\infty}\delta(\phi)$ and assume $\delta >0$.  
Then, the pure policy $\alpha^*$ and the distribution $\mu^*=\tilde\mu^{*, (x,\alpha^*(x))}$  introduced in Definition~\ref{MFCGS}  form a solution to the MFCG problem as in Definition~\ref{mfcg_definition}.
\end{theorem}
\begin{proof}
From the iterations described in \eqref{DEs} we have that the limit satisfies, for all $(x,a)$, 
\[
    Q^{*\phi}_{\mu^{*\phi}}(x,a)=f(x,a,\mu^{*\phi},\tilde\mu^{*\phi,(x,a)}_{Q^{*\phi},\mu^{*\phi}})+\gamma\sum_{x'}p(x'|x,a,\mu^{*\phi},\tilde\mu^{*\phi,(x,a)}_{Q^{*\phi},\mu^{*\phi}}) \min_{a'} Q^{*\phi}_{\mu^{*\phi}}(x,a'), 
\]
and by definition of $\alpha^*$, 
\[
    Q^{*\phi}_{\mu^{*\phi}}(x,a) > \min_{a'} Q^{*\phi}_{\mu^{*\phi}}(x,a')\quad \text{for}\quad a\neq \alpha^*(x).
\]
By definition of the action gap, the difference between the left-hand side and the right-hand side above is greater than the gap $\delta(\phi)$. Under the condition that the errors in 
\eqref{mfcg_errorphi} are small enough (by choosing  $\phi$ large enough), we can replace $Q^{*\phi}_{\mu^{*\phi}}$ by $Q^*$, $\tilde\mu^{*\phi}_{Q^{*\phi},\mu^{*\phi}}$ by $\tilde\mu^*$ and $\mu^{*\phi}$ by $\mu^*$ and obtain:
\[
Q^*(x,a) = f(x,a,\mu^{*},\tilde\mu^{*, (x,a)})+\gamma\sum_{x'}p(x'|x,a,\mu^{*},\tilde\mu^{*, (x,a)})\min_{a'} Q^*(x,a') > \min_{a'} Q^*(x,a')\,\, \text{for}\,\, a\neq \alpha^*(x),
\] 
i.e. $\argmin Q^*= \argmin Q^{*\phi}_{\mu^{*\phi}}=\alpha^*$, so that $(\alpha^*,\mu^*)$ is indeed the MFCG solution. Recall that by Remark~\ref{mu=mu}, for $a=\alpha^*(x)$, we have $\tilde\mu^{*,(x,\alpha^*(x))}=\mu^*$ and $Q^*(x,\alpha^*(x))=V^*(x)$, the optimal value function.
\end{proof}

\subsection{Synchronous Stochastic Approximation}
Recall in Section~\ref{sec:syncalgo}, we defined a synchronous algorithm with stochastic approximation. In this chapter, we will established the convergence of this three-timescale approach with stochastic approximation. Let us introduce the following notation,
\[
    \psi^\mu_n(x):=\sum_{m=1}^n \rho^\mu_m\mathbf{P}_m(x),
    \\\qquad
    \psi^{\mu^{(x,a)}}_n(y):=\sum_{m=1}^n \rho^{\tilde\mu}_m\mathbf{P}^{(x,a)}_m(y),
    \\\qquad
    \psi^Q_n(x,a):=\sum_{m=1}^n \rho^Q_m \mathbf{T}_m(x,a),
\]
where $\mathbf{P}_m(x)$, $\mathbf{P}^{(x,a)}_m(y)$ and $\mathbf{T}_m(x,a)$ are defined in \eqref{mdifference}.
\begin{proposition}\label{squaremartingale}
    Suppose  Assumption~\ref{squaresumlearningrate_mfcg} hold. $\psi^\mu_n(x)$, $\psi^{\mu^{(x,a)}}_n(y)$ and $\psi^Q_n(x,a)$ defined above, are square integrable martingales and hence they converge a.s. as $n \to +\infty$.
\end{proposition}
\begin{proof}
    By Lemma 4.5 in \cite{Borkar00}, we directly have $\psi^Q_n$ is a square integrable martingale. Also we have $\mathbbm{E}[\|\mathbf{{P}}_n\|^2]<1$ and $\mathbbm{E}[\|\mathbf{{P}}^{(x,a)}_n\|^2]<1$. Using this and the square summability of $(\rho^{\mu}_n)_n$ and $(\rho^{\tilde\mu}_n)_n$ assumed in Assumption~\ref{squaresumlearningrate_mfcg}, the bound immediately follows, which shows that $\psi^\mu_n$ and $\psi^{\mu^{(x,a)}}_n$ are also square integrable martingales. Then by martingale convergence theorem \cite[p.62]{neveu1975discrete}, we have the convergence of the three martingales.
\end{proof}

Based on the above result and the results proved in the previous subsection, we have the following theorem. 
\begin{theorem}\label{Stocastic MFCG}
    Suppose Assumptions \ref{mfcglip}, \ref{mfcglp}, \ref{SecondGASE}, \ref{ThirdGASE} and \ref{squaresumlearningrate_mfcg} hold. Then, $(\mu_n, Q_n,\tilde\mu_n^{(x,a)})$ obtained by the synchronous algorithm described in Section~\ref{sec:syncalgo} converge to $(\mu^{*\phi},Q^{*\phi}_{\mu^{*\phi}},\tilde\mu^{*\phi,(x,a)}_{Q^{*\phi},\mu^{*\phi}})$ a.s. as $n \to\infty$.
\end{theorem}
\begin{proof}
    With Assumptions \ref{mfcglip}, \ref{mfcglp}, \ref{SecondGASE}, \ref{ThirdGASE} and \ref{squaresumlearningrate_mfcg}, and the results of Propositions \ref{PLipschitz}, \ref{TPLipschitz}, \ref{TLipschitz}, \ref{emu'_MFCG}  and \ref{squaremartingale}, the assumptions of \cite[Theorem 1.1]{Borkar97} are satisfied. This result guarantees the convergence in the statement.
\end{proof}
\subsection{Stochastic Approximation and Asynchronous Setting for the Three-timescale System}
In the previous section, we assumed that learner have access to a generative model, i.e., to a simulator which can provide the samples of transitions drawn according to the hidden dynamic for arbitrary state $x$. However, as in our full Algorithm~\ref{algo:U3MFQL} a more realistic situation, the learner is constrained to follow the trajectory sampled by the environment without the ability to choose arbitrarily its state. Only one state-action pair is visited at each time step. We call this an \emph{asynchronous setting}. Let $\rho^{\tilde Q}_n$ be a short notation of $\rho^Q_{n,X^{(X_n,A_n)}_n,A^{(X_n,A_n)}_n}$. We define the system as follows for MFCG,
 \begin{subnumcases}{}
    \label{diffa}
        \mu^{(x,a)}_{n+1} = \mu^{(x,a)}_n + \rho_{n,X^{(x,a)}_n,A^{(x,a)}_n}^{\tilde\mu}(\mathcal{\tilde P}^{(x,a)}_3(\mu_n,Q_n,\mu_n^{(x,a)}) + \mathbf{P}^{(x,a)}_n)\\
        \label{diffb}
        \mu_{n+1} = \mu_n + \rho_n^{\mu}(\mathcal{P}_3(\mu_n, Q_n,\mu^{(x,a)}_n) + \mathbf{P}_n)\\
        \label{diffc}
        Q_{n+1}(X_n,A_n) = Q_n(X_n,A_n) + \rho^{\tilde Q}_n(\mathcal{T}_3(\mu_n, Q_n,\mu^{(X_n,A_n)}_n)+\mathbf{T}_n(X_n,A_n))\mathbbm{1}_{\{X^{(X_n,A_n)}_n=X_n\}}
\end{subnumcases}
for $x\in\mathcal{X},\:a\in\mathcal{A}, n\geq0$, where $(X^{(x,a)}_n,A^{(x,a)}_n)$ and $(X_n, A_n)$ are the state-action pair at time $n$ defined in Algorithm~\ref{algo:U3MFQL}. Comparing to the synchronous environment in the previous section, now we update the $Q$-table at most one state-action pair for each time. As a consequence, the state-action pairs are not all visited at the same frequency and the learning rate needs to be adjusted accordingly. We denote the learning rate for each state-action pair $(x,a)$ as $\rho^Q_{n,x,a}$. The martingale difference sequences are defined  as those in the Section~\ref{sec:syncalgo}, namely~\eqref{mdifference}.

Next, we will establish the convergence of the three-timescale approach with stochastic approximation under asynchronous setting defined in \eqref{diffa}--\eqref{diffc}. 

We will use the following assumption. 
\begin{assumption}\label{dlearningrate}
    For any $(x,a)$, the sequence $(\rho^{\tilde\mu}_{n,x,a})_{n\in \mathbbm{N}}$ satisfy the following conditions: 
     \[\sum_k \rho^{\tilde\mu}_{k,x,a} = \infty\quad \mbox{and} \quad \sum_k |\rho^{\tilde\mu}_{k,x,a}|^2<\infty.\]
\end{assumption}
\begin{assumption}[Ideal tapering stepsize]\label{asynchronouslearningrate}
    For any $(x,a)$, the sequence  $(\rho^Q_{n,x,a})_{n\in \mathbbm{N}}$ satisfies the following conditions: 
    \begin{itemize}
        \item (i) $\sum_k \rho^Q_{k,x,a} = \infty,$ and  $\sum_k |\rho^Q_{k,x,a}|^2<\infty.$
        \item (ii) $\rho^Q_{k+1,x,a}\leq \rho^Q_{k,x,a}$ from some k onwards.
        \item (iii) there exists $r\in(0,1)$ such that
        $
            \sum_{k}(\rho^Q_{k,x,a})^{(1+q)}<\infty, \forall q\geq r.
        $
        \item (iv) for $\xi\in(0,1)$, 
        $
            \sup_{k} \frac{\rho^Q_{[\xi k],x,a}}{\rho^Q_{k,x,a}}<\infty,
        $
        where $[\cdot]$ stands for the integer part.
        \item (v) for $\xi \in(0,1)$, and $D(k):=\sum_{\ell=0}^k \rho^Q_{\ell,x,a}$, 
        $
            \frac{D([yk])}{D(k)}\to 1,
        $
        uniformly in $y\in[\xi,1]$.
    \end{itemize}
\end{assumption}

In the asynchronous setting, we update step-by-step the Q table along the trajectory generated randomly. In order to ensure convergence, every state-action pair needs to be updated repeatedly. We introduce $\nu(x,a,n)=\sum_{m=0}^{n}\mathbbm{1}_{\{(X_m,A_m)=(x,a)\}}$ as the number of times the process $(X_n,A_n)$ visit state $(x,a)$ up to time $n$. 
The following assumptions are related to the (almost sure) good behavior of the number visits along trajectories. 
\begin{assumption}\label{frequentupdate}
    There exists a deterministic $\Delta > 0$ such that for all $(x,a)$, 
\begin{align*}
    \liminf_{n\to\infty} \frac{\nu(x,a,n)}{n}\geq\Delta\:\:\:\:\:a.s.
\end{align*}
Furthermore, letting
$
    N(n,\xi) = \min\left\{m>n:\sum_{k=n+1}^{m}\rho^Q_{k,x,a}>\xi\right\} 
$
for $\xi>0$, the limit 
\begin{align*}
    \lim_{n\to\infty}\frac{\sum_{k=\nu((x,a),n)}^{\nu((x,a),N(n,\xi))}\rho^Q_{k,x,a}}{\sum_{k=\nu((x',a'),n)}^{\nu((x',a'),N(n,\xi))}\rho^Q_{k,x',a'}}
\end{align*}
exists a.s. for all pairs $(x,a)$, $(x',a')$. 
\end{assumption}

\begin{assumption}\label{balanced}
    There exists $d_{ij}>0$, such that for each $i=(x,a)$ and $j=(x',a')$
    \begin{align*}
        \lim_{n\to\infty} \frac{\sum_{m=0}^n \rho^Q_{m,j}}{\sum_{m=0}^n \rho^Q_{m,i}}=d_{ij}.
    \end{align*}
\end{assumption}

\begin{proposition}\label{lyp}
    Suppose Assumptions~\ref{mfcglip}, \ref{mfcglp}, ~\ref{SecondGASE} and  ~\ref{ThirdGASE}  hold. %
    Then for every $\mu$, there exists a strict Lyapunov function for the ODE $\dot Q_t = \mathcal{T}_3(\mu,Q_t,{\tilde\mu^{*\phi,(x,a)}_{Q_t,\mu}})$. 
\end{proposition}
\begin{proof}
Let $\tilde{Q}$ be defined by  $\mathcal{T}_3(\mu,\tilde{Q},{\tilde\mu^{*\phi,(x,a)}_{\tilde{Q},\mu}}) = 0$. Let us define $V= \|Q-\tilde{Q}\|_\infty$. By \cite[Case 1, Page 844]{Borkar98}, $V$ is the strict Lyapunov function for the ODE $\dot Q_t = \mathcal{T}_3(\mu,Q_t,{\tilde\mu^{*\phi,(x,a)}_{Q_t,\mu}})$. 
\end{proof}
Then, we have the following theorem.
\begin{theorem}\label{Asyc MFG}
    Suppose Assumptions~\ref{mfcglip}, \ref{mfcglp}, 
    \ref{SecondGASE}, \ref{ThirdGASE}, 
    \ref{squaresumlearningrate_mfcg}, 
    \ref{dlearningrate}, \ref{asynchronouslearningrate}, \ref{frequentupdate} and \ref{balanced} hold. Then, $(\mu_n, Q_n,\tilde\mu_n^{(x,a)})$ defined in \eqref{diffa}--\eqref{diffc} converges to $(\mu^{*\phi},Q^{*\phi}_{\mu^{*\phi}},\tilde\mu^{*\phi,(x,a)}_{Q^{*\phi},\mu^{*\phi}})$ a.s. as $n \to\infty$.
\end{theorem}
\begin{proof}
It's not hard to check our learning rates defined in \eqref{diffa}--\eqref{diffc} satisfy \ref{squaresumlearningrate_mfcg}, \ref{dlearningrate}, \ref{asynchronouslearningrate}, \ref{frequentupdate} and \ref{balanced}. With the stated assumptions and the results of Propositions \ref{PLipschitz}, \ref{TPLipschitz}, \ref{TLipschitz}, \ref{emu'_MFCG}, \ref{squaremartingale} and \ref{lyp}, the assumptions of \cite[Theorem 3.1]{Borkar98} and \cite[Lemma 4.11]{Konda99} are satisfied. These results guarantee the convergence in the statement.
\end{proof}
\section{Numerical Illustration}\label{sec:example}

We present a very simple example with an explicit solution to the MFCG problem to illustrate the performance of our algorithm.

\subsection{The Model}
Let us consider the state space $\mathcal{X}=\{x_0=0,x_1=1\}$ and the action space $\mathcal{A}=\{stay=0,move=1\}$. The Markovian dynamics is given by
\[ 
    p(x'|x,a) = 
    \begin{cases}
        1-p, & \hbox{ if } x'=a(x) 
        \\
        p, & \hbox{ if } x'\ne a(x),
    \end{cases}
\] 
where $0<p<1$
represents a small noise parameter, and where as such we impose $0<p<0.5$. Note that the dynamics does not depend on the population distributions $\mu$ or $\tilde\mu$, and, therefore, in the notation of Assumption \ref{mfcglip}, $L^{glob}_p=L^{loc}_p=0$.

Let us consider the cost function:
\[
   f(x,a,\mu,\tilde\mu) = x + c_{g}\mu_0 +  c_{l}\tilde\mu_0,
\]
where $c_g>0$ and $c_l>0$, and $\mu_0$ (resp. $\tilde\mu_0$) is the probability to be at $x_0$ under the global distribution $\mu$ (resp. under the local distribution $\tilde\mu$). One may think of $\mu_0$ as $1-\bar{\mu}$, and $\tilde\mu_0$ as $1-\bar{\tilde\mu}$, that is as functions of the means. Note that 
\[
    |f(x,a,\mu,\tilde\mu) - f(x,a,\mu',\tilde\mu)|
    = c_g|\mu_0 - \mu'_0| 
    \le L^{glob}_f \|\mu - \mu'\|_1,
\]
\[
    |f(x,a,\mu,\tilde\mu) - f(x,a,\mu,\tilde\mu')|
    = c_l|\tilde\mu_0 - \tilde\mu'_0| 
    \le L^{loc}_f \|\tilde\mu - \tilde\mu'\|_1,
\]
with $L^{glob}_f=c_{g}/2$ and $L^{loc}_f= c_{l}/2$.

Recall that $c_{min}$ is defined in \eqref{def:cmin}. In this example we have:
\[
    c_{min}
    = \min_{x,x',a,\mu,\tilde\mu} p(x'|x,a,\mu,\tilde\mu)
    = p>0.
\]      
Therefore, in this example Assumptions~\ref{mfcglip} and \ref{mfcglp} are satisfied.

\subsection{Theoretical Solutions}
There are 4 pure policies $\alpha$ given by: 
$
\{(s,s),(s,m), (m,s), (m,m)\},
$
associated to the corresponding limiting distributions $\mu^\alpha$: 
$
\{(\frac{1}{2},\frac{1}{2}), (1-p,p), (p,1-p), (\frac{1}{2},\frac{1}{2})\}.
$

As in Section \ref{sec:Q-MFCG}, for a strategy $\alpha$, we define the strategy $\tilde{\alpha}^{(x,a)}$ by:
\begin{align*}
    \tilde{\alpha}^{(x,a)} (x') := \begin{cases}
        a\hspace{3mm} \text{if}\hspace{5mm} x'=x, \\
        \alpha(x')\hspace{2mm} \text{for}\hspace{3mm} x'\neq x.
    \end{cases} 
\end{align*}

For a given distribution $\mu$, the Bellman equation \eqref{MKV-Bellman} reads:
\[
Q^*_\mu(x,a)=f(x,a,\mu,\mu^{\tilde{\alpha^*}(x,a)})+\gamma\left[ (1-p)Q^*_\mu(a(x),s)\wedge Q^*_\mu(a(x),m) +p Q^*_\mu(1-a(x),s)\wedge Q^*_\mu(1-a(x),m)\right].
\]
When $c_{l}$ large enough ($c_l>2\gamma$), one can check that the optimal strategy is $\alpha^*=(m,s)$ which corresponds to the case 
\[
Q^*(0,m)<Q^*(0,s),\quad Q^*(1,s)<Q^*(1,m),
\]
and to the distribution $\mu^*=\mu^{\alpha^*}=(p,1-p)$.
The optimal Q-values are given by
\[
Q^*(0,m)=\frac{c_gp}{1-\gamma}+\frac{c_lp-\gamma p+\gamma}{1-\gamma}, \quad Q^*(1,s)=Q^*(0,m) +1,
\]
\[
Q^*(0,s)=\frac{c_gp}{1-\gamma}+\frac{c_l}{2}+\gamma\frac{c_lp-2\gamma p+\gamma +p}{1-\gamma},\quad Q^*(1,m)=Q^*(0,s)+1,
\]
so that $Q^*(0,s)-Q^*(0,m)=Q^*(1,m)-Q^*(1,s)=\frac{1}{2}(1-2p)(c_l-2\gamma)>0$ as soon as $c_l>2\gamma$.
Finally, note that $V^*(0)=Q^*(0,m)$ and $V^*(1)=Q^*(1,s)$.

Under our model,  for fixed $\mu$, the MFCG problem can be considered as a MFC problem. The global distribution $\mu$ will only shift the Q-table from the corresponding MFC problem by $\frac{c_g \mu_0}{1-\gamma}$. In \cite[Appendix D]{andrea23}, we proved that for such a model there is a unique optimal strategy $(m,s)$ for such MFC problem, so that we have unique solution to our MFCG problem. So far, we have proved Assumption is \ref{MFCGuniqueness} is satisfied.
\subsection{MFCG Assumptions \ref{SecondGASE}, \ref{ThirdGASE} and the gap in Theorem \ref{th:mainMFCG} }
We use a $\softmin$  policy $\pi$ defined by $\pi(a|x)={\softmin}_\phi Q(x)(a)$.
For any given $Q$,  the four limiting distributions (denoted $\tilde\mu^{*\phi, (x,a)}_{Q,\mu}$ in Proposition \ref{emu'_MFCG}) corresponding to the four state-action paris are characterized by:
\[
\tilde\mu_1^{(0,s)} = \frac{p}{1-(1-2p)\pi(s|1)}, \quad \tilde\mu_1^{(0,m)} =\frac{1-p}{2-2p-(1-2p)\pi(s|1)},
\]
\[
\tilde\mu_1^{(1,s)} = \frac{1-p - (1-2p)\pi(s|0)}{1 - (1-2p)\pi(s|0)}, \quad \tilde\mu_1^{(1,m)} =\frac{1-p - (1-2p)\pi(s|0)}{2-2p - (1-2p)\pi(s|0)}.
\]
Then, we compute the GASE $Q_\mu^{*\phi}$ in Assumption \ref{SecondGASE} by solving $\mathcal{T}_3(\mu,Q_\mu^{*\phi},\tilde\mu^{*\phi, (x,a)}_{Q_\mu^{*\phi},\mu}) = 0$. The uniqueness is proved by a similar argument as in \cite{andrea23} Section 5.2.4.. Their values are given by:
\begin{align*}
Q_\mu^{*\phi}(0,m) &= \frac{c_g\mu_0+c_l\tilde\mu_0^{{*\phi}(0,m)}+\gamma(1-p)(1-c_l\tilde\mu_0^{{*\phi}(0,m)} + c_l\tilde\mu_0^{{*\phi}(1,s)})}{1-\gamma}\\ Q_\mu^{*\phi}(1,s)&=Q_\mu^{*\phi}(0,m)+1-c_l\tilde\mu_0^{{*\phi}(0,m)} + c_l\tilde\mu_0^{{*\phi}(1,s)}\\
Q_\mu^{*\phi}(0,s) &= c_g\mu_0+c_l\tilde\mu_0^{{*\phi}(0,s)}+\gamma\frac{c_g\mu_0+c_l\tilde\mu_0^{{*\phi}(0,m)}+\gamma(1-p)(1-c_l\tilde\mu_0^{{*\phi}(0,m)} + c_l\tilde\mu_0^{{*\phi}(1,s)})}{1-\gamma}\\
&+\gamma p(1-c_l\tilde\mu_0^{{*\phi}(0,m)} + c_l\tilde\mu_0^{{*\phi}(1,s)})\\
Q_\mu^{*\phi}(1,m)&=Q_{\mu}^{*\phi}(0,s)+1-c_l\tilde\mu_0^{{*\phi}(0,s)} + c_l\tilde\mu_0^{{*\phi}(1,m)}
\end{align*}
Finally, we can compute the GASE $\mu^{*\phi}$ in Assumption \ref{ThirdGASE} by solving $\mathcal{P}_3 = 0$. Since we have two states, it is enough to write this equation at one point, say $x=1$. Since the global distribution $\mu$ affects Q only as an additive shift, the policy $\pi$ given by $Q_\mu^{*\phi}$ is independent of $\mu$, and we obtain the equation:

\begin{align*}
    \mu^{*\phi}_1 & = (1-\mu^{*\phi}_1)\left[\pi(s|0)p(1|0,s) + \pi(m|0)p(1|0,m)\right] + \mu^{*\phi}_1\left[\pi(s|1)p(1|1,s) + \pi(m|1)p(1|1,m)\right] \\
          &=  (1-\mu^{*\phi}_1)\left[p\pi(s|0)+ (1-p)\pi(m|0)\right] + \mu^{*\phi}_1\left[ (1-p)\pi(s|1)+p\pi(m|1)\right]\\
          &= (p\pi(s|0)+ (1-p)\pi(m|0)) - \mu^{*\phi}_1(p\pi(s|0)+ (1-p)\pi(m|0)) + \mu^{*\phi}_1((1-p)\pi(s|1)+p\pi(m|1))\\
          &= (p\pi(s|0)+ (1-p)\pi(m|0)) - \mu^{*\phi}_1(p\pi(s|0)+ (1-p)\pi(m|0)-(1-p)\pi(s|1)-p\pi(m|1))
\end{align*}
This is a linear equation which has a unique solution. One can also prove this is a GASE because $L(\mu_1)=(\mu_1-\mu^{*\phi}_1)^2$ is 
a Lyapunov function for the ODE $\dot{\mu_1} = \mathcal{P}_3(\mu_1)$ since $L'( \mu_1)g(\mu_1)<0$ and $L'(\mu^{*\phi}_1)g(\mu^{*\phi}_1)=0$. Repeating the argument with $\mu^{*\phi}_0=1-\mu^{*\phi}_1$, one deduces that $\mu^{*\phi}$ is the unique GASE in Assumption \ref{ThirdGASE}.

Regarding the gap in Theorem \ref{th:mainMFCG}, we have:
\[\delta(\phi) = \min\left(Q^{*\phi}_{\mu^{*\phi}}(0,m)-Q^{*\phi}_{\mu^{*\phi}}(0,s),Q^{*\phi}_{\mu^{*\phi}}(1,s)-Q^{*\phi}_{\mu^{*\phi}}(1,m)\right)
\]
and $\lim_{\phi\to\infty} \delta(\phi) = Q^{*}(0,m)-Q^{*}(0,s) = Q^{*}(1,s)-Q^{*}(1,m) = \frac{1}{2}(1-2p)(c_l-2\gamma)>0$ as soon as $c_l>2\gamma$ and $p<0.5$.

\subsection{Numerical Results}\label{nr}
We consider the problem defined above with the choice of parameters: $N =2$ states, $p=0.1$, $c_l=c_g=5$, $ \phi=500$, $ \gamma=0.5$. 
We run the deterministic Algorithm \eqref{DEs} with $(\omega^{\tilde\mu},\omega^Q,\omega^{\mu}) = (0.55,0.75,0.95)$. This choice satisfies the assumptions that we used to prove convergence to the MFCG solution. Figure~\ref{odemfcg} illustrates the result.
We provide the convergence plot of $\mu_n(x_0)$, i.e., the value of the distribution at state 0, as a function of the step $n$. The y-axis is the value of $\mu_n(x_0)$, $\mu^{\alpha^*}_n(x_0)$ and x-axis is the number of iterations.  

\begin{figure}[H]
\center 
\subfloat[MFCG: $(\omega^{\tilde\mu},\omega^Q,\omega^{\mu}) = (0.55,0.75,0.95)$]{\includegraphics[width=0.5\textwidth]{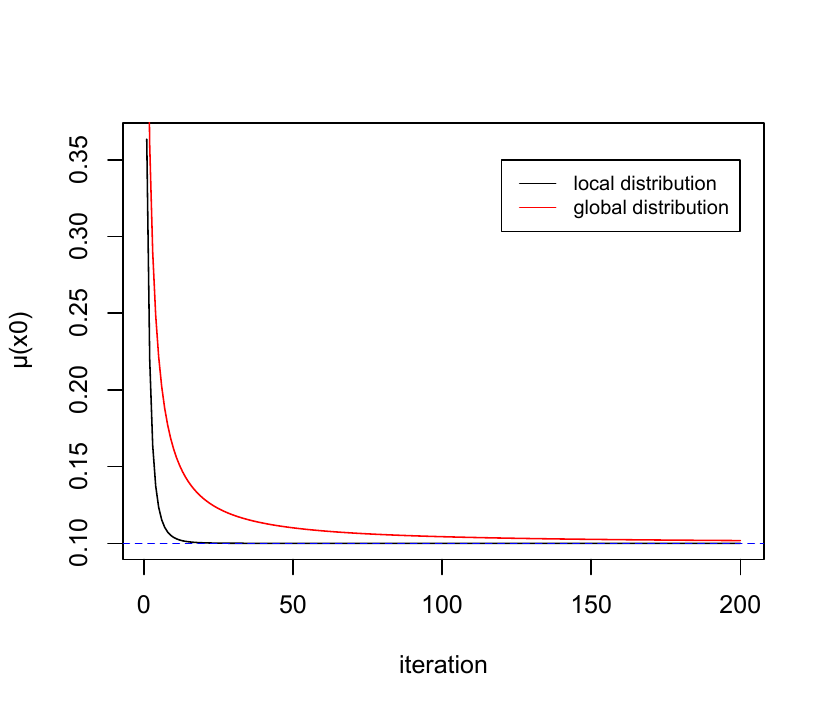}}
\caption{MFCG setting: Convergence of the distribution. The plot represents the value of $\mu_n(x_0)$ and $\mu^{\alpha^*}_n(x_0)$ as a function of $n$.}
\label{odemfcg}
\end{figure}
The resulting distribution $\mu=\mu^{\alpha^*}=(\mu(x_0),\mu(x_1))$ which is about  $(0.1,0.9)=(p,1-p)$, corresponds to the equilibrium distribution in the  MFCG scenario with optimal strategy $\alpha^*=(m,s)$ given by the limiting Q-table:

\begin{center}
\begin{tabular}{ |p{3cm}||p{2cm}|p{2cm}|p{2cm}|p{2cm}|  }
 \hline
 & 
 \multicolumn{2}{|c|}{Q-table} 
 &                                            %
\multicolumn{2}{|c|}{Theoretical Values} \\
 \hline
 States/Actions & Action 0 &Action 1 & Action 0 & Action 1\\
 \hline
 State 0  &4.501 & 2.901  &4.5 &2.9\\
 \hline 
 State 1  & 3.901& 5.501 &3.9& 5.5    \\

\hline
\end{tabular}
\end{center}
We  also run the algorithms \ref{algo:U3MFQL} for this example with the same choice of learning rates. Figure~\ref{fig:example-mfcg1} illustrates the result.
We provide the convergence plot of $\mu_n(x_0)$, i.e., the value of the distribution at state 0, as a function of the step $n\leq 1000K$. The y-axis is the value of $\mu_n(x_0)$, $\mu^{\alpha^*}_n(x_0)$ and x-axis is the number of iterations.  

\begin{figure}[H]
\center 
\subfloat[MFCG: $(\omega^{\tilde\mu}, \omega^Q, \omega^{\mu})=(0.55,0.75, 0.95)$, first $2$K iterations]
{\includegraphics[width=0.4\textwidth]{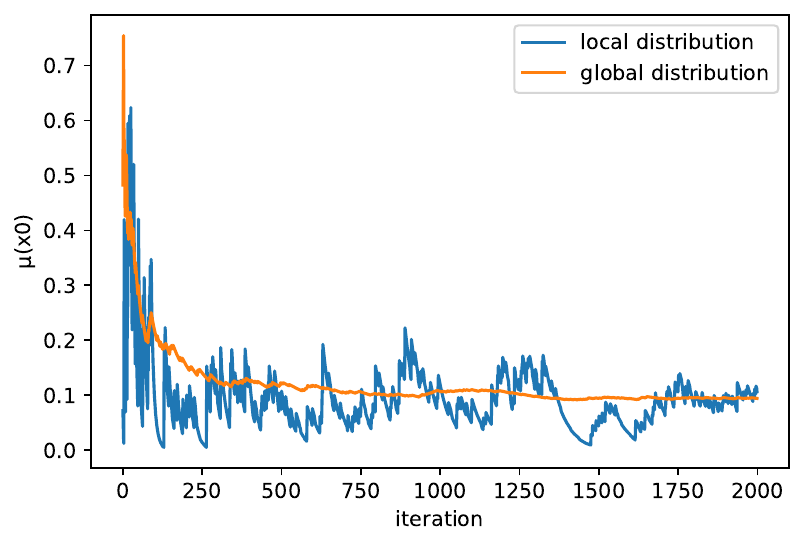}\label{fig:MFCG}}\qquad\qquad
\subfloat[MFCG: $(\omega^{\tilde\mu}, \omega^Q, \omega^{\mu})=(0.55,0.75, 0.95)$, for $1000$K iterations]{\includegraphics[width=0.4\textwidth]{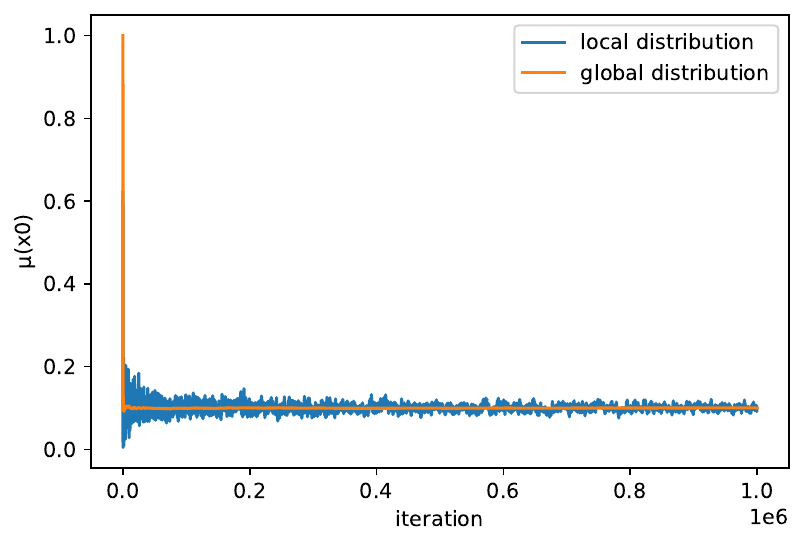}\label{fig:MFCG1}}
\caption{MFCG setting: Convergence of the distribution. The plot represents the value of $\mu_n(x_0)$ and $\mu^{\alpha^*}_n(x_0)$ as a function of $n$.  \label{fig:example-mfcg1}}
\end{figure}

As expected, figure \ref{fig:MFCG} shows that local distribution converges faster than the global distribution but have more fluctuations since local distribution is updated faster than the global distribution. Figure \ref{fig:MFCG1} shows eventually they converge perfectly. 

The resulting distribution $\mu=\mu^{\alpha^*}=(\mu(x_0),\mu(x_1))$ which is about  $(0.1,0.9)=(p,1-p)$, corresponds to the equilibrium distribution in the  MFCG scenario with optimal strategy $\alpha^*=(m,s)$ given by the limiting Q-table:

\begin{center}
\begin{tabular}{ |p{3cm}||p{2cm}|p{2cm}|p{2cm}|p{2cm}|  }
 \hline
 & 
 \multicolumn{2}{|c|}{Q-table} 
 &                                            %
\multicolumn{2}{|c|}{Theoretical Values} \\
 \hline
 States/Actions & Action 0 &Action 1 & Action 0 & Action 1\\
 \hline
 State 0  &4.334 & 2.907  &4.5 &2.9\\
 \hline 
 State 1  & 3.913& 5.477 &3.9& 5.5    \\

\hline
\end{tabular}
\end{center}

\bibliographystyle{apalike}
\bibliography{bibtex}

\begin{thebibliography}{}

\bibitem[Angiuli et~al., 2022a]{andrea22}
Angiuli, A., Detering, N., Fouque, J.-P., Lauri{\`e}re, M., and Lin, J.
  (2022a).
\newblock Reinforcement learning algorithm for mixed mean field control games.
\newblock {\em arXiv:2205.02330}.

\bibitem[Angiuli et~al., 2022b]{Andrea20}
Angiuli, A., Fouque, J.-P., and Lauri{\`e}re, M. (2022b).
\newblock Unified reinforcement q-learning for mean field game and control
  problems.
\newblock {\em Mathematics of Control, Signals, and Systems}, 34:217--271.

\bibitem[Angiuli et~al., 2023]{andrea23}
Angiuli, A., Fouque, J.-P., Lauri{\`e}re, M., and Zhang, M. (2023).
\newblock Convergence of multiscale reinforcement q-learning algorithms for
  mean field game and control problems.
\newblock {\em arXiv:2312.06659}.

\bibitem[Bensoussan et~al., 2013]{MR3134900}
Bensoussan, A., Frehse, J., and Yam, S. C.~P. (2013).
\newblock {\em Mean field games and mean field type control theory}.
\newblock Springer Briefs in Mathematics. Springer, New York.

\bibitem[Borkar and Soumyanatha, 1997]{563625}
Borkar, V. and Soumyanatha, K. (1997).
\newblock An analog scheme for fixed point computation. i. theory.
\newblock {\em IEEE Transactions on Circuits and Systems I: Fundamental Theory
  and Applications}, 44(4):351--355.

\bibitem[Borkar, 1997]{Borkar97}
Borkar, V.~S. (1997).
\newblock Stochastic approximation with two time scales.
\newblock {\em Systems \& Control Letters}, 29(5):291--294.

\bibitem[Borkar, 1998]{Borkar98}
Borkar, V.~S. (1998).
\newblock Asynchronous stochastic approximations.
\newblock {\em SIAM Journal on Control and Optimization}, 36(3):840--851.

\bibitem[Borkar and Meyn, 2000]{Borkar00}
Borkar, V.~S. and Meyn, S.~P. (2000).
\newblock The o.d.e. method for convergence of stochastic approximation and
  reinforcement learning.
\newblock {\em SIAM Journal on Control and Optimization}, 38(2):447--469.

\bibitem[Carmona and Delarue, 2018]{carmona2018probabilisticI-II}
Carmona, R. and Delarue, F. (2018).
\newblock {\em Probabilistic Theory of Mean Field Games with Applications
  I-II}.
\newblock Springer.

\bibitem[Carmona et~al., 2019]{CarmonaLauriereTan-2019-LQMFRL}
Carmona, R., Lauri{\`e}re, M., and Tan, Z. (2019).
\newblock Linear-quadratic mean-field reinforcement learning: Convergence of
  policy gradient methods.
\newblock Preprint.

\bibitem[Carmona et~al., 2023]{carmona2023model}
Carmona, R., Lauri{\`e}re, M., and Tan, Z. (2023).
\newblock Model-free mean-field reinforcement learning: mean-field mdp and
  mean-field q-learning.
\newblock {\em The Annals of Applied Probability}, 33(6B):5334--5381.

\bibitem[Cui and Koeppl, 2021]{cui2021approximately}
Cui, K. and Koeppl, H. (2021).
\newblock Approximately solving mean field games via entropy-regularized deep
  reinforcement learning.
\newblock In {\em International Conference on Artificial Intelligence and
  Statistics}, pages 1909--1917. PMLR.

\bibitem[Elie et~al., 2020]{elie2020convergence}
Elie, R., Perolat, J., Lauri{\`e}re, M., Geist, M., and Pietquin, O. (2020).
\newblock On the convergence of model free learning in mean field games.
\newblock In {\em in proc. of AAAI}.

\bibitem[Frikha et~al., 2023]{frikha2023actor}
Frikha, N., Germain, M., Lauri{\`e}re, M., Pham, H., and Song, X. (2023).
\newblock Actor-critic learning for mean-field control in continuous time.
\newblock {\em arXiv preprint arXiv:2303.06993}.

\bibitem[Gu et~al., 2021]{gu2021meanQ}
Gu, H., Guo, X., Wei, X., and Xu, R. (2021).
\newblock Mean-field controls with q-learning for cooperative marl: convergence
  and complexity analysis.
\newblock {\em SIAM Journal on Mathematics of Data Science}, 3(4):1168--1196.

\bibitem[Guo et~al., 2019]{guo2019learning}
Guo, X., Hu, A., Xu, R., and Zhang, J. (2019).
\newblock Learning mean-field games.
\newblock In {\em Advances in Neural Information Processing Systems}, pages
  4966--4976.

\bibitem[Huang et~al., 2006]{MR2346927}
Huang, M., Malham{\'e}, R.~P., and Caines, P.~E. (2006).
\newblock Large population stochastic dynamic games: closed-loop
  {M}c{K}ean-{V}lasov systems and the {N}ash certainty equivalence principle.
\newblock {\em Commun. Inf. Syst.}, 6(3):221--251.

\bibitem[Konda and Borkar, 1999]{Konda99}
Konda, V.~R. and Borkar, V.~S. (1999).
\newblock Actor-critic--type learning algorithms for markov decision processes.
\newblock {\em SIAM Journal on Control and Optimization}, 38(1):94--123.

\bibitem[Lasry and Lions, 2007]{MR2295621}
Lasry, J.-M. and Lions, P.-L. (2007).
\newblock Mean field games.
\newblock {\em Jpn. J. Math.}, 2(1):229--260.

\bibitem[Lauri{\`e}re et~al., 2022]{lauriere2022learning}
Lauri{\`e}re, M., Perrin, S., Geist, M., and Pietquin, O. (2022).
\newblock Learning mean field games: A survey.
\newblock {\em arXiv preprint arXiv:2205.12944}.

\bibitem[Motte and Pham, 2019]{motte2019mean}
Motte, M. and Pham, H. (2019).
\newblock Mean-field markov decision processes with common noise and open-loop
  controls.
\newblock {\em arXiv preprint arXiv:1912.07883}.

\bibitem[Neveu, 1975]{neveu1975discrete}
Neveu, J. (1975).
\newblock {\em Discrete-parameter Martingales}.
\newblock North-Holland mathematical library. North-Holland.

\bibitem[Sell, 1973]{SELL197342}
Sell, G.~R. (1973).
\newblock Differential equations without uniqueness and classical topological
  dynamics.
\newblock {\em Journal of Differential Equations}, 14(1):42--56.

\bibitem[Subramanian and Mahajan, 2019]{SubramanianMahajan-2018-RLstatioMFG}
Subramanian, J. and Mahajan, A. (2019).
\newblock Reinforcement learning in stationary mean-field games.
\newblock In {\em Proceedings. 18th International Conference on Autonomous
  Agents and Multiagent Systems}.

\bibitem[Sutton and Barto, 2018]{sutton2018reinforcement}
Sutton, R.~S. and Barto, A.~G. (2018).
\newblock {\em Reinforcement learning: An introduction}.
\newblock MIT press.

\bibitem[Tembine, 2017]{tembine2017mean}
Tembine, H. (2017).
\newblock Mean-field-type games.
\newblock {\em AIMS Math}, 2(4):706--735.

\bibitem[Wang et~al., 2020]{wang2020reinforcement}
Wang, H., Zariphopoulou, T., and Zhou, X.~Y. (2020).
\newblock Reinforcement learning in continuous time and space: A stochastic
  control approach.
\newblock {\em Journal of Machine Learning Research}, 21(198):1--34.

\bibitem[Wang and Zhou, 2020]{wang2020continuous}
Wang, H. and Zhou, X.~Y. (2020).
\newblock Continuous-time mean--variance portfolio selection: A reinforcement
  learning framework.
\newblock {\em Mathematical Finance}, 30(4):1273--1308.

\bibitem[Watkins, 1989]{watkins1989learning}
Watkins, C. J. C.~H. (1989).
\newblock {\em Learning from delayed rewards}.
\newblock PhD thesis, King's College, Cambridge.

\end{thebibliography}

\appendix
\section{GASE by Contraction Argument}\label{gasecontraction}
To prove that the GASE exist for the second and third O.D.E. in (\ref{DEs}), we need to control the parameter $\phi$ properly to have a strict contraction. To that end, we introduce the following assumption.
\begin{assumption}\label{controlphicg}
    Assume $$
        0<\phi < \min\left(\frac{|\mathcal{X}|c_{min}-L_{p}^{total}}{|\mathcal{A}|}\frac{1-\gamma}{L_{f}^{total} + \frac{\gamma}{1-\gamma} L_{p}^{total}\|f\|_\infty}, \frac{(|\mathcal{X}|c_{min}-L_{p}^{glob})^2}{\phi|\mathcal{A}|(|\mathcal{X}|c_{min} -L_{p}^{glob}+ L_{p}^{loc})}\frac{1-\gamma}{L_{f}^{glob} + \frac{\gamma}{1-\gamma} L_{p}^{glob}\|f\|_\infty}\right),
    $$
    where $L_p^{total} = L_p^{glob} + L_p^{loc}$ and $L_f^{total} = L_f^{glob} + L_f^{loc}$.
\end{assumption}
 With this assumption, we have $$\gamma + (L_{f}^{loc} + \gamma L_{p}^{loc}\|Q\|_\infty)\frac{\phi|\mathcal{A}|}{|\mathcal{X}|c_{min}-L_{p}^{loc}} <1,$$ and  $$\frac{\phi |\mathcal{A}|(|\mathcal{X}|c_{min} -L_{p}^{glob}+ L_{p}^{loc})}{|\mathcal{X}|c_{min}-L_{p}^{glob}}\frac{L_f^{glob} + \gamma L_p^{glob}\|Q^*_{\mu}\|_\infty}{1-\gamma} + L_{p}^{glob} + 1-|\mathcal{X}|c_{min}<1,$$ which will give the strict contraction property we need in the next propositions. 
\begin{proposition}\label{eQ_MFCG}
    Suppose Assumption \ref{mfcglip} and \ref{controlphicg} hold. Then for any given $\mu$, $\dot{Q}_t = \mathcal{T}_3(\mu,Q_t,{\mu^{*\phi}_{Q_t,\mu}})$ has a unique GASE, that we will denote by $Q^{*\phi}_\mu$.
\end{proposition}
\begin{proof}

    We show that $Q \mapsto \mathcal{B}_{\mu,\mu^{*\phi}_Q}Q$ is a strict contraction. We have:
    \begin{align*}
        \|\mathcal{B}_{\mu,\mu^{*\phi}_Q}Q - \mathcal{B}_{\mu,\mu^{*\phi}_{Q'}}Q'\|_\infty &\leq \|\mathcal{B}_{\mu,\mu^{*\phi}_Q}Q - \mathcal{B}_{\mu,\mu^{*\phi}_{Q}}Q'\|_\infty+\|\mathcal{B}_{\mu,\mu^{*\phi}_Q}Q' - \mathcal{B}_{\mu,\mu^{*\phi}_{Q'}}Q'\|_\infty \\
        &\leq \gamma \|Q-Q'\|_\infty + (L_f^{loc} + \gamma L_p^{loc}\|Q\|_\infty)\|\mu^{*\phi}_Q - \mu^{*\phi}_{Q'}\|_1 \\
        &\leq \gamma \|Q-Q'\|_\infty + (L_f^{loc} + \gamma L_p^{loc}\|Q\|_\infty)\frac{\phi|\mathcal{A}|}{|\mathcal{X}|c_{min}-L_p}\|Q-Q'\|_\infty\\
        &\leq \left(\gamma + (L_f^{loc} + \gamma L_p^{loc}\|Q\|_\infty)\frac{\phi|\mathcal{A}|}{|\mathcal{X}|c_{min}-L_p^{loc}}\right)\|Q-Q'\|_\infty,
    \end{align*}
    where we used the bound derived in the proof of Proposition~\ref{emu'_MFCG}. 
    This shows the strict contraction property. As a result, by the contraction mapping theorem \cite{SELL197342}, a unique GASE exists and furthermore by \cite[Theorem 3.1]{563625}, it is the limit of $Q_t$.
\end{proof}

\begin{proposition}\label{emu_MFCG}
    Under assumptions \ref{mfcglip} and \ref{controlphicg}, $\dot{\mu}_t = \mathcal{P}_3(\mu_t,Q^{*\phi}_{\mu_t},\tilde\mu^{*\phi}_{Q^{*\phi}_{\mu_t},\mu_t})$ has a unique GASE, that we will denote by $\mu^{*\phi}$.
\end{proposition}
\begin{proof}
We show that $\mu \mapsto \mathrm{P}^{\softmin_\phi Q^*_{\mu},\mu,\tilde\mu^{*\phi}_{Q^{*\phi}_{\mu},\mu}}\mu$ is a strict contraction. We have, denoting $\tilde\mu^{*\phi}_{Q^{*\phi}_{\mu},\mu}$ as $\mu'_\mu$ for the sake of brevity:
    \begin{align*}
        &\|\mathrm{P}^{\softmin_\phi Q^*_{\mu}, \mu, \mu'_\mu}\mu - \mathrm{P}^{\softmin_\phi Q^*_{\mu_1}, \mu_1, \mu'_{\mu_1}}\mu_1\|_1 
        \\
        &\leq |\mathrm{P}^{\softmin_\phi Q^*_{\mu}, \mu,\mu'_\mu}\mu - \mathrm{P}^{\softmin_\phi Q^*_{\mu'}, \mu,\mu'_{\mu_1}}\mu\|_1+ \|\mathrm{P}^{\softmin_\phi Q^*_{\mu}, \mu,\mu'_{\mu_1}}\mu - \mathrm{P}^{\softmin_\phi Q^*_{\mu_1}, \mu,\mu'_{\mu_1}}\mu\|_1 \\&\hspace{5mm}+\|\mathrm{P}^{\softmin_\phi Q^*_{\mu_1}, \mu,\mu'_{\mu_1}}\mu - \mathrm{P}^{\softmin_\phi Q^*_{\mu_1}, \mu_1,\mu'_{\mu_1}}\mu_1\|_1 
        \\
        &\leq \|\mathcal{P}_3(\mu, Q^*_{\mu},\mu'_\mu) - \mathcal{P}_3(\mu, Q^*_{\mu},\mu'_{\mu_1})\|_1 +\|\mathcal{P}_3(\mu, Q^*_{\mu},\mu'_{\mu_1}) - \mathcal{P}_3(\mu, Q^*_{\mu_1},\mu'_{\mu_1})\|_1 
        \\
        &\qquad +\|\mathcal{P}_3(\mu, Q^*_{\mu},\mu'_{\mu_1}) - \mathcal{P}_3(\mu_1, Q^*_{\mu_1},\mu'_{\mu_1})\|_1
        \\
        &\leq L_{p}^{loc}\|\mu'_{\mu}-\mu'_{\mu_1}\|_1+\|\mathcal{P}_3(\mu, Q^*_{\mu},\mu'_{\mu_1}) - \mathcal{P}_3(\mu, Q^*_{\mu_1},\mu'_{\mu_1})\|_1 +(1+L_{p}^{glob}-|\mathcal{X}|c_{min})\|\mu-\mu_1\|_1\\
        &\leq\frac{L_{p}^{loc}\phi|\mathcal{A}|}{|\mathcal{X}|c_{min}-L_p^{glob}}\|Q^*_{\mu}-Q^*_{\mu_1}\|_\infty+ \phi|\mathcal{A}|\|Q^*_{\mu}-Q^*_{\mu_1}\|_\infty +(1+L_{p}^{glob}-|\mathcal{X}|c_{min})\|\mu-\mu_1\|_1\\
        &\leq \left(\frac{\phi |\mathcal{A}|(|\mathcal{X}|c_{min} -L_{p}^{glob}+ L_{p}^{loc})}{|\mathcal{X}|c_{min}-L_{p}^{glob}}\frac{L_{f}^{glob} + \gamma L_{p}^{glob}\|Q^*_{\mu}\|_\infty}{1-\gamma} + L_{p}^{glob} + 1-|\mathcal{X}|c_{min}\right)\|\mu-\mu_1\|_1,
    \end{align*}
    which shows the strict contraction property. As a result, by contraction mapping theorem \cite{SELL197342}, a unique GASE exists and furthermore by \cite[Theorem 3.1]{563625}, it is the limit of $\mu_t$.    
\end{proof}

\section{Proof of Lemma \ref{musystem}}\label{prooflemma}
\begin{proof}
Recall system \eqref{mus}:
\begin{align*}
    \begin{cases}
        \tilde\mu^{(x,a)}=\tilde\mu^{(x,a)} \tilde{\mathrm{P}}^{\alpha, \mu, \tilde\mu^{(x,a)}}_{(x,a)}\\
        \mu=\mu\mathrm{P}^{\alpha, \mu, \tilde\mu^{(x,\alpha(x))}}.
    \end{cases}
\end{align*}
Firstly, for fixed $\mu$, by a similar argument as in Proposition \ref{emu'_MFCG},  $\tilde\mu^{(x,a)}\mapsto \tilde\mu^{(x,a)} \tilde{\mathrm{P}}^{\alpha, \mu, \tilde\mu^{(x,a)}}_{(x,a)}$ is a strict contraction, which guarantees that $\tilde\mu^{(x,a)}=\tilde\mu^{(x,a)} \tilde{\mathrm{P}}^{\alpha, \mu, \tilde\mu^{(x,a)}}_{(x,a)}$ has unique solution. Then, we can deduce that, for all $x$, $\tilde\mu^{(x,\alpha(x))}$ satisfies 
\begin{align}\label{localeq}
    \tilde\mu^{(x,\alpha(x))}=\tilde\mu^{(x,\alpha(x))} \tilde{\mathrm{P}}^{\alpha, \mu, \tilde\mu^{(x,\alpha(x))}}_{(x,\alpha(x))} = \sum_{x'}\tilde\mu^{(x,\alpha(x))}(x')p(\cdot|x',\alpha(x'),\mu,\tilde\mu^{(x,\alpha(x))}).
\end{align}
Let us denote $\tilde\mu^{(x,\alpha(x))}$ by $\nu^x$. We want to show that $\nu^x$ is the same for all $x$. By definition in \eqref{def:P3T3}, we can show that \eqref{localeq} can be written as $\nu^x = \nu^x\mathrm{P}^{\alpha, \mu, \nu^x}$. Since $\nu^x \mapsto \nu^x\mathrm{P}^{\alpha, \mu, \nu^x}$ is a strict contraction,  we have a unique solution for $\nu^x = \nu^x\mathrm{P}^{\alpha, \mu, \nu^x}$ and, consequently, $\nu^x$ is independent of $x$. Now the second equation in \eqref{mus} is well-defined. 

Next we want to show that the solution to $\eqref{mus}$ is unique. Let us consider  the following system first:
\begin{align}\label{subs}
    \begin{cases}
        \nu^x = \nu^x\mathrm{P}^{\alpha, \mu, \nu^x}\\
        \mu=\mu\mathrm{P}^{\alpha, \mu, \nu^x}
    \end{cases}
\end{align}
The above system \eqref{subs} is the system  \eqref{mus} with $a=\alpha(x)$, and we can show it has a unique solution as follows. Assuming $\nu'^{x}$ and $\mu'$ also satisfy \eqref{subs}, then
\begin{align*}
    \|\nu^x-\nu'^x\|_1+\|\mu-\mu'\|_1 &= \|\nu^x\mathrm{P}^{\alpha, \mu, \nu^x}- \nu'^x\mathrm{P}^{\alpha, \mu', \nu'^x}\|_1 + \|\mu\mathrm{P}^{\alpha, \mu, \nu_x}-\mu'\mathrm{P}^{\alpha, \mu', \nu'_x}\|_1\\
    &\leq \|\nu^x\mathrm{P}^{\alpha, \mu, \nu^x}- \nu'^x\mathrm{P}^{\alpha, \mu, \nu'^x}\|_1+\|\nu'^x\mathrm{P}^{\alpha, \mu, \nu'^x}- \nu'^x\mathrm{P}^{\alpha, \mu', \nu'^x}\|_1 \\
    &\quad \|\mu\mathrm{P}^{\alpha, \mu, \nu_x}-\mu'\mathrm{P}^{\alpha, \mu', \nu_x}\|_1+\|\mu'\mathrm{P}^{\alpha, \mu', \nu_x}-\mu'\mathrm{P}^{\alpha, \mu', \nu'_x}\|_1\\
    &\leq (L_p^{loc}+1-|\mathcal{X}|c_{min})\|\nu^x-\nu'^x\|_1 + L_p^{glob}\|\mu-\mu'\|_1 \\
    &\quad + (L_p^{glob}+1-|\mathcal{X}|c_{min})\|\mu-\mu'\|_1 + L_p^{loc}\|\nu^x-\nu'^x\|_1\\
    &\leq (2L_p^{loc}+1-|\mathcal{X}|c_{min})\|\nu^x-\nu'^x\|_1 + (2L_p^{glob}+1-|\mathcal{X}|c_{min})\|\mu-\mu'\|_1
\end{align*}
which implies that
\[
(|\mathcal{X}|c_{min}-2L_p^{loc})\|\nu^x-\nu'^x\|_1 + (|\mathcal{X}|c_{min}-2L_p^{glob})\|\mu-\mu'\|_1\leq 0.
\]
By Assumption \ref{mfcglp}, we have $(|\mathcal{X}|c_{min}-2L_p^{loc})>0$ and $(|\mathcal{X}|c_{min}-2L_p^{glob})>0$, and consequently we deduce $ \|\nu^x-\nu'^x\|_1=\|\mu-\mu'\|_1=0$, so that $\nu^x=\nu'^x$ and $\mu=\mu'$, which proves  that \eqref{subs} has at most one solution. 

Then, we can show that there exists a unique solution to \eqref{mus}. Assume that there exists another solution $(\tilde\mu^{''(x,a)},\mu'')$ with $\mu''\neq\mu$ satisfying \eqref{mus}. Consequently,  $(\tilde\mu^{''(x,\alpha(x))},\mu'')$ is a solution to \eqref{subs} which is a contradiction with the uniqueness for \eqref{subs}. As a result, there is a unique $\mu$ solving \eqref{mus}, and then by the strict contraction property of $\tilde\mu^{(x,a)}\mapsto \tilde\mu^{(x,a)} \tilde{\mathrm{P}}^{\alpha, \mu, \tilde\mu^{(x,a)}}_{(x,a)}$, uniqueness of $\tilde\mu^{(x,a)}$ is obtained. As a result, the system \eqref{mus} has unique solution $(\tilde\mu^{(x,a)},\mu)$.

Finally, we want to show that $\mu=\nu^x$.  Since $\nu \mapsto \nu\mathrm{P}^{\alpha, \nu, \nu}$ is a strict contraction,  the equation $\nu = \nu\mathrm{P}^{\alpha, \nu, \nu}$ has a unique and well-defined solution $\nu$. Then, one can check that $\mu=\nu^x=\nu$ is a solution of \eqref{subs} and conclude by its uniqueness property. 
\end{proof}

\end{document}